\documentclass[11pt]{article}
  \usepackage{a4wide}
  \usepackage{amsmath}
  \usepackage{amssymb}
  \usepackage{latexsym}
  \usepackage[pdftex]{graphicx}
  \usepackage{color}
  \usepackage[mathscr]{euscript}

  \newenvironment{proof}{\vspace{1ex}\noindent{\bf Proof:}}{\hspace*{\fill}$\blacksquare$\vspace{1ex}}
  \newenvironment{proofof}[1]{\vspace{1ex}\noindent{\bf Proof of #1:}}{\hspace*{\fill}$\blacksquare$\vspace{1ex}}

  \newtheorem{theorem}{Theorem} 
  \newtheorem{lemma} [theorem] {Lemma}
  \newtheorem{conjecture} [theorem] {Conjecture}

\newcommand{\Bcal}[0]{\ensuremath{{\mathcal B}}}
\newcommand{\Ccal}[0]{\ensuremath{{\mathcal C}}}
\newcommand{\Dcal}[0]{\ensuremath{{\mathcal D}}}

\newcommand{\Pcal}[0]{\ensuremath{{\mathcal P}}}
\newcommand{\Qcal}[0]{\ensuremath{{\mathcal Q}}}

\newcommand{\Xcal}[0]{\ensuremath{{\mathcal X}}}
\newcommand{\Ycal}[0]{\ensuremath{{\mathcal Y}}}

\newcommand{\eR}[0]{\ensuremath{ \mathbb R}}

\newcommand{\eN}[0]{\ensuremath{ \mathbb N}}
\newcommand{\Zed}[0]{\ensuremath{ \mathbb Z}}

\newcommand{\Vtil}[0]{\tilde{V}}

\newcommand{\Xtil}[0]{\tilde{X}}

\newcommand{\Cscr}[0]{\ensuremath{{\mathscr C}}}

\newcommand{\Rscr}[0]{\ensuremath{{\mathscr R}}}

\newcommand{\Pee}[0]{\ensuremath{{\mathbb P}}}

\newcommand{\Ee}[0]{\ensuremath{{\mathbb E}}}

\newcommand{\isd}[0]{\hspace{.2ex} \raisebox{-.1ex}{$=$} \hspace{-1.5ex} 
\raisebox{1ex}{{$\scriptstyle d$}} \hspace{.8ex} }

 \newcommand{\eps}{\varepsilon}

\DeclareMathOperator{\Po}{Po}

\DeclareMathOperator{\dd}{d}
\DeclareMathOperator{\Var}{Var}

\newcommand{\GPo}[0]{\ensuremath{G_{\text{Po}}}}

\newcommand{\Gammatil}[0]{\ensuremath{\tilde{\Gamma}}}

\newcommand{\lambdac}[0]{\ensuremath{\lambda_{\text{crit}}}}
\newcommand{\nuc}[0]{\ensuremath{\nu_{\text{crit}}}}

\title{Law of large numbers for the largest component in a hyperbolic model of complex networks}
\author{Nikolaos Fountoulakis\thanks{Research supported by a Marie Curie Career Integration Grant
 PCIG09-GA2011-293619.}\\
\small School of Mathematics\\[-0.8ex]
\small University of Birmingham\\[-0.8ex] 
\small United Kingdom\\
\small\tt n.fountoulakis@bham.ac.uk\\
\and
Tobias M\"uller\thanks{Research partially supported by a Vidi grant from Netherlands Organisation for Scientific Research (NWO)} \\
\small Mathematical Institute\\[-0.8ex]
\small Utrecht University\\[-0.8ex]
\small  The Netherlands\\
\small\tt t.muller@uu.nl
}

\begin{document}

\maketitle

\begin{abstract} 
We consider the component structure of a recent model of random graphs on the hyperbolic plane that was introduced 
by Krioukov et al. 
The model exhibits a power law degree sequence, small distances and clustering, features that are associated with so-called
complex networks. The model is controlled by two parameters $\alpha$ and $\nu$ where, roughly speaking, $\alpha$ controls the 
exponent of the power law and $\nu$ controls the average degree.
Refining earlier results, we are able to show a law of large numbers for the largest component.
That is, we show that the fraction of points in the largest component tends in probability to a constant $c$ that depends only on
$\alpha,\nu$, while all other components are sublinear. We also study how $c$ depends on $\alpha, \nu$. 
To deduce our results, we introduce a local approximation of the random graph by a continuum percolation model on 
$\mathbb{R}^2$ that may be of independent interest. 
\medskip

\noindent
\texttt{keywords}: random graphs, hyperbolic plane, giant component, law of large numbers

\noindent
\emph{2010 AMS Subj. Class.}: 05C80, 05C82; 60D05, 82B43
\end{abstract}


\section{Introduction}
The component structure of random graphs and in particular the size of the largest component has been a central problem in the 
theory of random graphs as well as in percolation theory. 
Already from the founding paper of random graph theory~\cite{ErRenConn} the emergence of the \emph{giant component} is a 
central theme that recurs, through the development of more and more sophisticated results.    

In this paper, we consider a random graph model that was introduced recently by Krioukov et al. in~\cite{ar:Krioukov}.
The aim of that work was the development of a geometric framework for the analysis of properties of the so-called 
\emph{complex networks}. This term describes a diverse class of networks that emerge in a range of human activities or 
biological processes and includes social networks, scientific collaborator networks as well as computer networks, such as the Internet, and the 
power grid -- see for example~\cite{BarAlb}.  
These are networks that consist of a very large number of heterogeneous nodes (nowadays social networks such as the Facebook 
or the Twitter have billions of users), and they are sparse. 
However, locally they are dense - this is the \emph{clustering phenomenon} which makes it more likely for two 
vertices that have a common neighbour to be connected.  
Furthermore, these networks are \emph{small worlds}: 
almost all pairs of vertices that are in the same component are within a 
short distance from each other. 
But probably the most strikingly ubiquitous property they have is that their degree distribution is \emph{scale free}. 
This means that its tail follows a \emph{power law}, usually with exponent between 2 and 3 (see for example~\cite{BarAlb}). 
Further discussion on these characteristics can be found in the books of Chung and Lu~\cite{ChungLuBook+} and of Dorogovtsev~\cite{Dor}.

In the past decade, several models have appeared in the literature aiming at capturing these features. 
Among the first was the \emph{preferential attachment model}. 
This term describes a class of models of randomly growing graphs whose aim is to capture 
the following phenomenon: nodes which are already popular retain their popularity or tend to become more popular as 
the network grows. 
It was introduced by Barab\'asi and Albert~\cite{BarAlb} and subsequently defined and studied rigorously 
by Bollob\'as, Riordan and co-authors (see for example~\cite{DegSeq},~\cite{Diam}).

Another extensively studied model was defined by Chung and Lu~\cite{ChungLu1+}, \cite{ChungLuComp+}.
In some sense, this is a generalisation of the standard binomial model $G(n,p)$. 
Each vertex is equipped with a weight, which effectively corresponds to its expected degree, and every two vertices are 
joined~\emph{independently}
of every other pair with probability that is proportional to the product of their weights. 
If the distribution of these weights follows a power law, then it turns out that the degree distribution of the resulting random graph follows a power law as well. 
This model is a special case of an \emph{inhomogeneous random graph} of rank 1~\cite{BolJanR}. 

All these models have their shortcomings and none of them incorporates \emph{all} the above features.
For example, the Chung--Lu model exhibits a power law 
degree distribution (provided the weights of the vertices are suitably chosen)  and average distance of order $O(\log \log N)$ (when
the exponent of the power law is between 2 and 3, see~\cite{ChungLu1+}), but it is locally tree like (around most vertices) 
and therefore it does not exhibit clustering. This is also the situation in the Barab\'asi-Albert model. 

For the the Chung-Lu model, this is the case as the pairs of vertices form edges independently. 
But for clustering to appear, it has to be the case that for two edges that share an endvertex the probability that their other two
endvertices are joined must be higher compared to that where we assume nothing about these 
edges. 
This property is naturally present in random graphs that are created over metric spaces, such as random geometric graphs. 
In a random geometric graph, the vertices are a random set of points on a given metric space, with any two of them being 
adjacent if their distance is smaller than some threshold. 

The model of Krioukov et al.~\cite{ar:Krioukov} does exactly this. It introduces a geometric framework on the theory of complex networks and it is based on the hypothesis that hyperbolic geometry underlies these networks. 
In fact, it turns out that the power-law degree distribution emerges naturally from the underlying hyperbolic geometry. 
They defined an associated random graph model, which we will describe in detail in the next section, and considered some of its
typical properties. 
More specifically, they observed a power-law degree sequence as well as clustering properties. 
These characteristics were later verified rigorously by Gugelmann et al.~\cite{ar:Gugel} as well as
by the second author~\cite{ar:Foun13+} and Candellero and the second author~\cite{Candellero} (these two papers are on a 
different, but closely related model).

The aim of the present work is the study the component structure of such a random graph. More specifically, we consider the number of vertices that are contained in a largest component of the graph. 
In previous work~\cite{BFMgiantEJC} with M.~Bode, we have determined the range of the parameters, in which the so-called \emph{giant component} 
emerges. We have shown that in this model this range essentially coincides with the range in which the exponent of the power 
law is smaller than 3. What is more, when the exponent of the power law is larger than 3, the random graph typically consists of many relatively small components, no matter how 
large the average degree of the graph is. 
This is in sharp contrast with the classical Erd\H{o}s-R\'enyi model (see~\cite{Bol01} or~\cite{JLR}) as well as with
the situation of random geometric graphs on Euclidean spaces (see~\cite{bk:Penrose}) where the dependence on the average degree is crucial. 

In the present paper, we give a complete description of this range and, furthermore, we show that the order of the largest 
connected component follows a law of large numbers. Our proof is based on the local approximation of the random graph model 
by an infinite continuum percolation model on $\mathbb{R}^2$, which may be of independent interest. We show that percolation 
in this model coincides with the existence of a giant component in the model of Krioukov et al. 
We now proceed with the definition of the model.

\subsection{The Krioukov-Papadopoulos-Kitsak-Vahdat-Bogu\~n\'a model} 

Let us recall very briefly some facts about the hyperbolic plane $\mathbb H$.
The hyperbolic plane is an unbounded 2-dimensional manifold of constant negative curvature $-1$.
There are several representations of it, such as the upper half-plane model, the Beltrami-Klein disk model and the Poincar\'e disk model. 
The Poincar\'e disk model is defined if we equip the unit disk ${\mathbb D} := \{(x,y) \in \mathbb{R}^2 \ : x^2+y^2 <1 \}$
with the metric given\footnote{%
Thus, the length of a curve $\gamma : [0,1]\to{\mathbb D}$ is given by 
$2 \int_0^1\frac{\sqrt{(\gamma_1'(t))^2+(\gamma_2'(t))^2}}{1-\gamma_1^2(t)-\gamma_1^2(t)}{\dd}t$.
} by the differential form ${\dd}s^2 = 4~\frac{dx^2 + dy^2}{(1-x^2-y^2)^2}$.
 For an introduction to hyperbolic geometry and the properties of the hyperbolic plane, the reader may refer to the book of Stillwell~\cite{bk:Stillwell}.

A very important property that differentiates the hyperbolic plane from the euclidean is the rate of growth of volumes.  
In $\mathbb H$, a disk of radius $r$ (i.e., the set of points at hyperbolic distance at most $r$ to a given point) has area equal to 
$2\pi(\cosh(r) - 1)$ and length of circumference equal to $2\pi\sinh(r)$.  
Another basic geometric fact that we will use later in our proofs is the so-called {\em hyperbolic cosine rule}. This
states that if $A,B,C$ are distinct points on the hyperbolic plane, and we denote by $a$ the distance between $B,C$, by $b$ the distance between 
$A,C$, by $c$ the distance between $A,B$ and by $\gamma$ the angle (at $C$) between the shortest $AC$- and $BC$-paths, then it holds 
that 
$\cosh(c) = \cosh(a)\cosh(b) - \cos(\gamma)\sinh(a)\sinh(b)$.

We can now introduce the model we will be considering in this paper.
We call it the Krioukov-Papadopoulos-Kitsak-Vahdat-Bogu\~n\'a-model, after its inventors, but for convenience we will 
abbreviate this to KPKVB-model.
The model has three parameters: the number of vertices $N$, which we think of as the parameter that tends to infinity, 
and $\alpha, \nu > 0$ which are fixed (that is, do not depend on $N$).
Given $N, \nu, \alpha$, we set $R := 2\log(N/\nu)$. 
Consider the Poincar\'e disk representation and let $O$ be the origin of the disk. 
We select $N$ points independently at random from the disk of (hyperbolic) radius $R$ centred at $O$, which we denote 
by $\Dcal_{R}$, according to the following probability distribution. 
If the random point $u$ has polar coordinates $(r, \theta)$, then $\theta, r$ are independent, 
$\theta$ is uniformly distributed in $(0,2\pi]$ and 
the cumulative distribution function of $r$ is given by:

\begin{equation}\label{eq:cdf}
 F_{\alpha,R}(r) = \left\{\begin{array}{cl}
 0 & \text{ if $r < 0$, } \\
 \frac{\cosh(\alpha r) - 1}{\cosh(\alpha R) - 1} & \text{ if } 0 \leq r \leq R, \\
 1 & \text{ if $r > R$ }
\end{array} \right.
\end{equation}

\noindent
Note that when $\alpha = 1$, then this is simply the uniform distribution on $\Dcal_R$. 
We label the random points as $1,\ldots, N$. 

It can be seen that the above probability distribution corresponds precisely to the polar coordinates of a point taken uniformly at random from the disk of radius $R$ around the origin in the hyperbolic plane of curvature\footnote{%
That is the natural analogue of the hyperbolic plane, in which the Gaussian curvature equals $-\alpha^2$ at every point.
One way to obtain (a model of) the the hyperbolic plane of curvature $-\alpha^2$ is to multiply the differential form in the 
Poincar\'e disk model by a factor $1/\alpha^2$. 
} $-\alpha^2$. 
Indeed, a simple calculation shows that the area of disc of radius $r$ on the hyperbolic plane of curvature 
$-\alpha^2$ is equal to $\cosh (\alpha r) -1$. 
So the above distribution can be viewed as selecting the $N$ points uniformly on a disc of radius $R$ on the hyperbolic 
plane of curvature $-\alpha^2$ and then projecting them onto the hyperbolic plane of curvature -1, preserving 
polar coordinates. This is where we create the random graph. 
The set of the $N$ labeled points will be the vertex set of our random graph and we denote it by $\mathrm{V}$. 
The KPKVB-random graph, denoted $G (N; \alpha, \nu)$, is formed when we join each pair of vertices, if and only if they are within 
(hyperbolic) distance $R$. Note this is precisely the radius of the disk $\Dcal_R$ that the points live on.
Figures~\ref{fig:KPKVB},\ref{fig:KPKVB_1} show examples of such a random graph on $N=500$ vertices.
\begin{figure}[h]
\centering

\includegraphics[scale = 0.27]{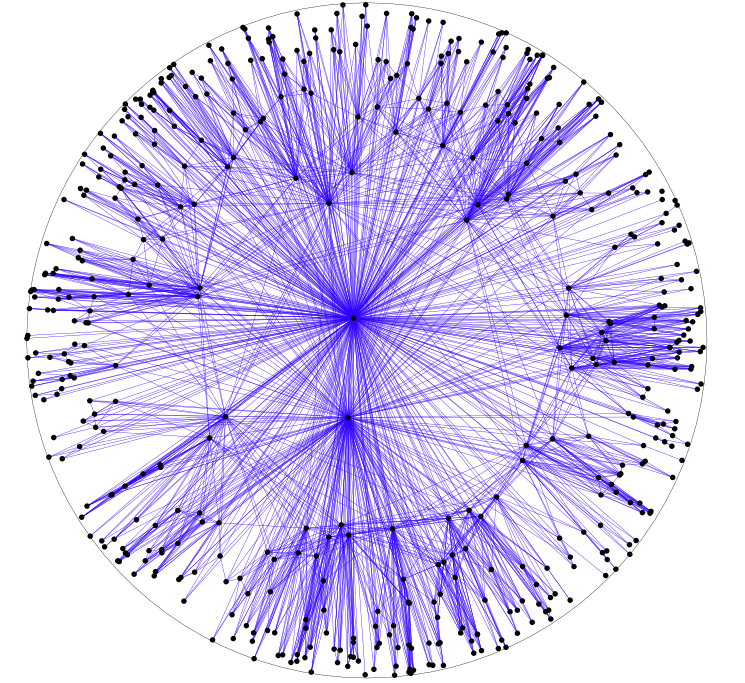}
\includegraphics[scale = 0.27]{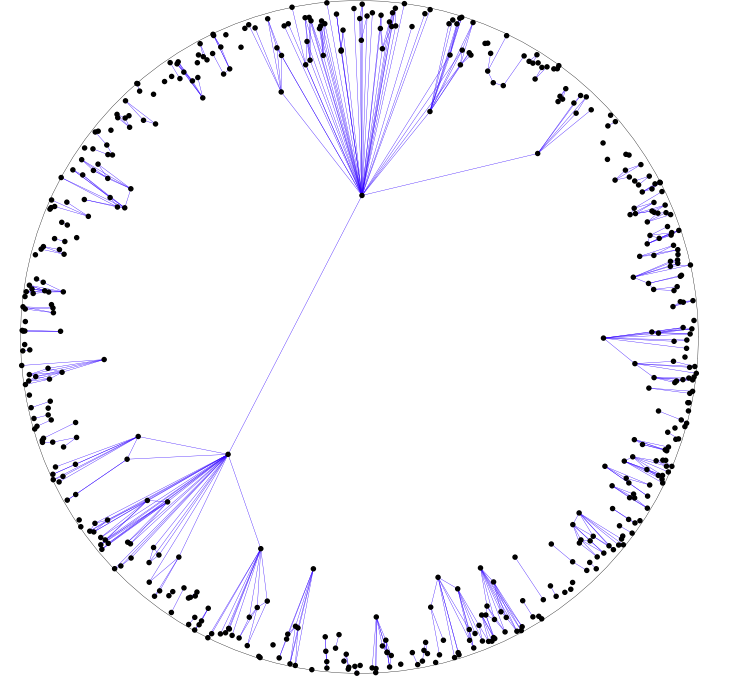}
\caption{The random graph $G(N;\alpha, \nu)$ with $N=500$ vertices, $\nu =2$ and $\alpha = 0.7$ and $3/2$. } 
\label{fig:KPKVB}
\end{figure}

\begin{figure}[h]
\centering
\includegraphics[scale = 0.27]{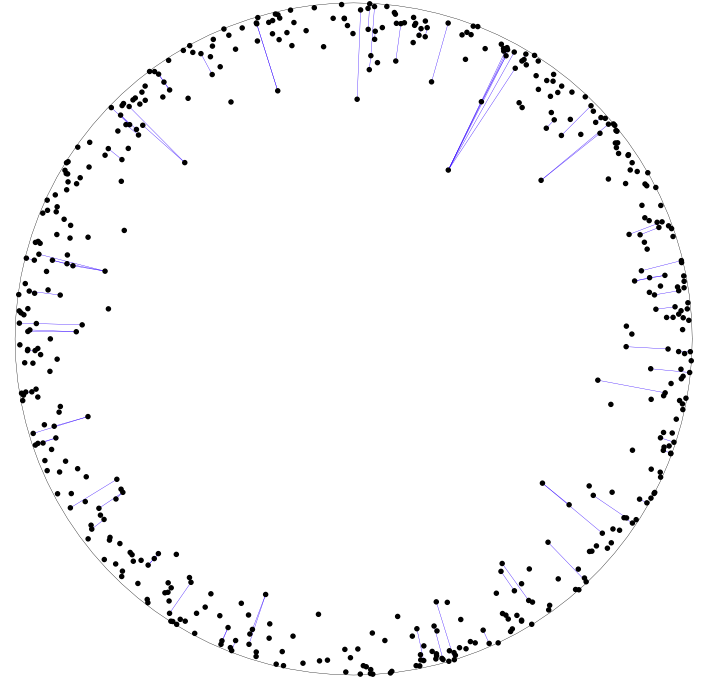}
\includegraphics[scale = 0.27]{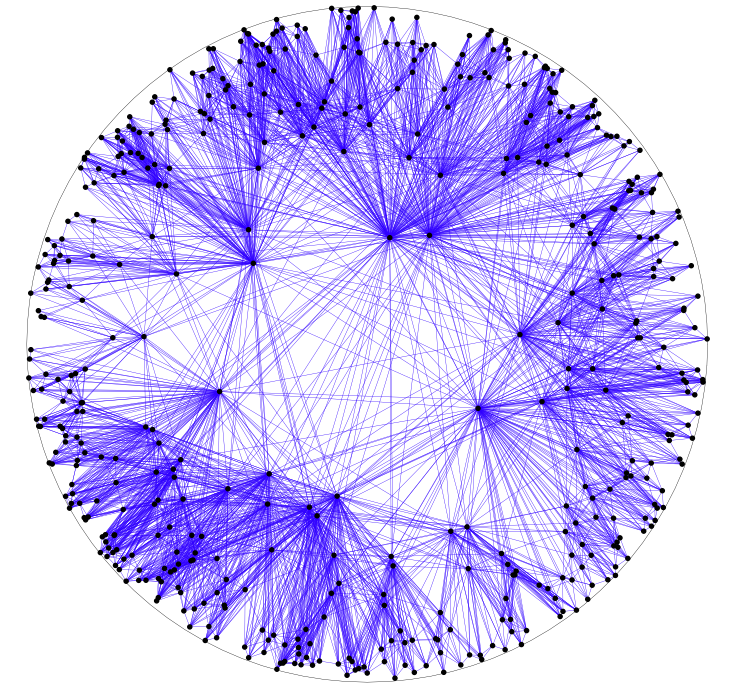}
\caption{The random graph $G(N;\alpha, \nu)$ with $N=500$ vertices, $\alpha =1$ and $\nu = 0.2$ and $10$.
 } \label{fig:KPKVB_1}
\end{figure}


\noindent
Krioukov et al.~\cite{ar:Krioukov} in fact defined a generalisation of this model where the probability distribution over the 
graphs on a given point set $\mathrm{V}$ of size $N$ is a Gibbs distribution which involves a temperature parameter. 
The model we consider in this paper corresponds to the limiting case where the temperature tends to 0.

Let us remark that the {\em edge-set of $G(N;\alpha,\nu)$ is decreasing in $\alpha$ and increasing in $\nu$} in the following sense.
We remind the reader that a {\em coupling} of two random objects $X,Y$ is a common probability space for a pair of objects $(X',Y')$ 
whose marginal distributions satisfy $X'\isd X, Y'\isd Y$.

\begin{lemma}[\cite{BFMgiantEJC}]
Let $\alpha, \alpha', \nu, \nu' > 0$ be such that $\alpha \geq \alpha'$ and $\nu \leq \nu'$.
For every $N \in \eN$, there exists a coupling such that $G(N;\alpha, \nu)$ is a subgraph of $G(N;\alpha',\nu')$.
\end{lemma}

\noindent
The proof can be found in our earlier paper~\cite{BFMgiantEJC} that was joint with Bode.

We should also mention that Krioukov et al.~in fact had an additional parameter in their definition of the model. 
However, it turns out that this parameter is not necessary in the following sense.
Every probability distribution (on labelled graphs) that is defined by some choice of parameters in the model with one extra parameter coincides with the probability distribution $G(N;\alpha,\nu)$ for some $N, \alpha, \nu$.
This is Lemma 1.1 in~\cite{BFMgiantEJC}.

Krioukov et al.~\cite{ar:Krioukov} focus on the degree distribution of $G (N; \alpha, \nu)$, showing that when 
$\alpha > {1\over 2}$ this follows a power law with exponent $2\alpha  + 1$. They also discuss clustering on the generalised version we briefly described above. Their results on the degree distribution have been verified rigorously by 
Gugelmann et al.~\cite{ar:Gugel}. 

Note that when $\alpha = 1$, that is, when the $N$ vertices are uniformly distributed in $\Dcal_R$, the exponent of the 
power law is equal to 3. When ${1\over 2} <\alpha < 1$, the exponent is between 2 and 3, as is the case in a number of networks 
that emerge in applications such as computer networks, social networks and biological networks (see for example~\cite{BarAlb}). 
When $\alpha = {1\over 2}$ then the exponent becomes equal to 2. This case has recently emerged in theoretical cosmology~\cite{ar:Krioukov2}. In a quantum-gravitational setting, networks between space-time events are considered 
where two events are connected (in a graph-theoretic sense) if they are causally connected, that is, one is located in the light cone of the other. The analysis of Krioukov et al.~\cite{ar:Krioukov2} indicates that the tail of the degree distribution follows a power law with exponent 2.    

As observed by Krioukov et al.~\cite{ar:Krioukov} and rigorously proved by Gugelmann et al.~\cite{ar:Gugel}, the average degree of the random graph can be ``tuned" through the parameter $\nu$ : 
for $\alpha > \frac12$, the average degree tends to  $2\alpha^2 \nu / \pi(\alpha-\frac12)^2$ in probability. 

In~\cite{BFMgiantEJC}, together with M.~Bode we have showed that $\alpha = 1$ is the critical point for the emergence of a 
``giant'' component in $G (N;\alpha, \nu)$. More specifically, we have showed that when $\alpha < 1$, the largest component 
contains a positive fraction of the vertices asymptotically with high probability, whereas if $\alpha > 1$, then the largest 
component is sublinear. For $\alpha = 1$, then the component structure depends on $\nu$. We show that if $\nu$ is large enough, 
then a ``giant'' component exists with high probability, whereas if $\nu$ is small enough, then with high probability all
components have sublinear size.   

Kiwi and Mitsche~\cite{KiwiMit} showed that when $\alpha < 1$ then the second largest component is sublinear 
with high probability. 

In~\cite{BFMconnRSA}, together with M.~Bode we studied the probability of being connected.
We showed that when $\alpha < \frac12$, then $G(N;\alpha,\nu)$ is connected with high probability, while if
$\alpha > \frac12$ the graph is disconnected with high probability. 
For $\alpha = \frac12$, the probability of connectivity tends to a constant $c=c(\nu)$ that depends on $\nu$.
Curiously, $c(\nu)$ equals one for $\nu \geq \pi$ while it increases strictly from zero to one as $\nu$ varies from $0$ to $\pi$.

\subsection{Results}
The results of this paper elaborate on the component structure of the KPKVB model, refining our previous results~\cite{BFMgiantEJC} with Bode. 
We show that the size of the largest component rescaled by $N$ (that is, the fraction 
of the vertices that belong to the largest component) converges in probability to a non-negative constant. We further determine this 
constant as the integral of a function that is associated with a percolation model on $\mathbb{R}^2$.
We also show that the number of points in the second largest connected component is sublinear. 
When $\alpha = 1$, we show that there exists a critical $\nu$ around which the constant changes from 0 to being positive. 
In other words, there is a critical value for $\nu$ around which the emergence of the giant component occurs. 
We let $\Cscr_{(1)}$ and $\Cscr_{(2)}$ denote the largest and second largest connected component (if the two have the 
same number of vertices, then the ordering is lexicographic using a labelling of the vertex set). 

\begin{theorem}\label{thm:main}
For every $\nu, \alpha > 0$ there exists a $c=c(\alpha,\nu)$ such that
\[ 
 \frac{|\Cscr_{(1)}|}{N} \to c \quad \quad \text{ and } \quad \quad \quad \frac{|\Cscr_{(2)}|}{N} 
 \to 0 \quad \quad \text{ in probability, }
\]
as $N\to\infty$. The function $c(\alpha,\nu)$ has the following properties:
\begin{enumerate}
 \item\label{itm:main.i} If $\alpha > 1$ and $\nu$ arbitrary then $c=0$;
 \item\label{itm:main.ii} If $\alpha < 1$ and $\nu$ arbitrary then $c > 0$;
 \item\label{itm:main.iii}There exists a \emph{critical value} $0 < \nuc < \infty$ such that if $\alpha=1$ then
 \begin{enumerate}
  \item if $\nu < \nuc$ then $c=0$;
  \item if $\nu > \nuc$ then $c>0$.
 \end{enumerate}
 \item\label{itm:main.iv} $c=1$ if and only if $\alpha \leq \frac12$.
 \item\label{itm:main.v} $c$ is continuous on $(0,\infty)^2 \setminus \{1\}\times [\nuc,\infty)$, and every point of
 $\{1\}\times (\nuc,\infty)$ is a point of discontinuity.
 \item\label{itm:main.vi} $c$ is strictly decreasing in $\alpha$ on $(\frac12,1)\times(0,\infty)$ and strictly increasing 
 in $\nu$ on $(\frac12,1)\times(0,\infty) \cup \{1\}\times(\nuc,\infty)$.
\end{enumerate}
\end{theorem}
The above theorem leaves open the case where $\alpha =1$ and $\nu = \nuc$. We do not know whether $c(1,\nuc)$ is positive or 0. 
The latter would imply that $c(1,\nu)$ is continuous as a function of $\nu$. 
We however conjecture that $c(1,\nuc) > 0$.

\subsubsection*{Notation} For a sequence of real-valued 
random variables $X_n$ defined on a sequence of probability spaces, 
we write $X_n \stackrel{p}{\to} a$, where $a\in \mathbb{R}$, to denote that 
the sequence $X_n$ converges to $a$ \emph{in probability}. For a sequence of positive real numbers $a_n$, we write 
$X_n = o_{p}(a_n)$, if $|X_n|/a_n \stackrel{p}{\to} 0$, as $n \to \infty$.  

An event occurs \emph{almost surely (a.s.)} if its probability is equal to 1. 
For a sequence of probability spaces $(\Omega_n, \mathbb{P}_n, \mathcal{F}_n)$, we say that a sequence of measurable sets
$\mathcal{E}_n \in \mathcal{F}_n$ occurs \emph{asymptotically almost surely (a.a.s.)}, if $\lim_{n \to \infty} 
\mathbb{P}_n(\mathcal{E}_n)=1$. 

\subsubsection{Tools and outline}

Roughly speaking, one can show that two vertices of radii $r_1$ and $r_2$, respectively, which are close to the periphery of $\Dcal_R$ are
adjacent if their relative angle is at most $e^{(y_1 + y_2)/2}/N$, where $y_1 = R - r_1$ and $y_2 = R - r_2$. Hence, conditional on $r_1$
and $r_2$, the probability that these two vertices are adjacent is proportional to $e^{(y_1 + y_2)/2}/N$. 

We will map $\Dcal_R$ onto $\mathbb{R}^2$ preserving $y$ and projecting the angle onto the $x$-axis scaling it by $N$. 
Hence, relative angle $e^{(y_1 + y_2)/2}/N$ between two vertices in $\Dcal_R$ will project to two points in $\mathbb{R}^2$ whose 
$x$-coordinates differ by at most $e^{(y_1 + y_2)/2}$. 

This gives rise to a continuum percolation model on $\mathbb{R}^2$, whose vertices are the points of an
inhomogeneous Poisson point process and any two of them are joined if their positions satisfy the above condition. 
This continuum percolation model is a good approximation of $G(N;\alpha, \nu)$ close to the periphery of $\Dcal_R$. 
As we shall see later in Section~\ref{sec:transfer}, we can couple the two using 
the aforementioned mapping so that the two graphs coincide close to the $x$-axis.

In Section~\ref{sec:cont_perc}, we determine the 
critical conditions for the parameters of the continuum percolation model that ensure the almost sure existence of an infinite component. 
Thereafter, in Section~\ref{sec:transfer}, we show how does this infinite model approximate $G(N;\alpha, \nu)$ 
and deduce the law of large numbers for the size of its largest and the second largest connected component. 
The fraction of vertices that are contained in the largest component converges in probability to a constant $c(\alpha, \nu)$. 
More specifically, we show that the function $c(\alpha,\nu)$ is given as the integral of the probability that a given point 
``percolates", that is, it belongs to an infinite component, in the infinite model.

\section{A continuum percolation model}\label{sec:cont_perc}

Given a countable set of points $P = \{p_1,p_2, \dots \} \subset \eR \times [0,\infty)$ in the upper
half of the (euclidean) plane, we define a graph $\Gamma(P)$ with vertex set $P$ by setting

\[
p_ip_j \in E(\Gamma (P)) \text{ if and only if } |x_i-x_j| < e^{\frac12(y_i+y_j)}=:t(y_i,y_j),
\]
\noindent
where we write $p_i = (x_i, y_i)$.
For a point $p= (x,y)$, we let $\Bcal (p)$ denote the \emph{ball} around $p$, that is, the set 
$\{ p'=(x',y') \ : \ y' > 0, \ |x-x'| < t(y,y')  \}$. Thus, $\Bcal (p)$ contains all those points that would potentially 
connect to $p$. 
During our proofs, we will need the notion of the \emph{half-ball} around 
$p$. More specifically, we denote by $\Bcal^+ (p)$ the set $\{ p'=(x',y') \ : \ y' > 0, \  0 <x' - x < t(y,y')  \}$ and 
by $\Bcal^- (p)$ the set $\{ p'=(x',y') \ : \ y' > 0, \ 0< x-x' < t(y,y')  \}$. In other words, the $\Bcal^+(p)$ consists 
of those points in $\Bcal (p)$ that are located on the right of $p$ (i.e., they have larger $x$-coordinate) and 
$\Bcal^- (p)$ consists of those points that are on the left of $p$. 
Finally, for a point $p \in \mathbb{R}^2$, we let $x(p)$ and $y(p)$ be its $x-$ and $y-$coordinates, respectively.
We first make a few easy geometric observations that we will rely on in the sequel.

\begin{lemma}\label{lem:cross}
The following hold for the graph $\Gamma(P)$ defined above.
\begin{enumerate}
 \item\label{itm:cross.i} If $p_ip_j \in E(\Gamma(P))$ and $p_k$ is {\em above} the line segment
$[p_i, p_j]$ (i.e.~$[p_i,p_j]$ intersects the vertical line though $p_k$ below $p_k$), 
then at least one of the edges $p_kp_i, p_kp_j$ is also present in $\Gamma(P)$;
 \item\label{itm:cross.ii} If $p_ip_j, p_kp_\ell \in E(\Gamma(P))$ and the line segments $[p_i, p_j], [p_k,p_\ell]$ cross, then
 at least one of the edges $p_ip_k, p_ip_\ell, p_jp_k, p_jp_\ell$ is also present in $\Gamma(P)$.
 \end{enumerate}
\end{lemma}

\begin{proofof}{part~\ref{itm:cross.i}}
By symmetry, we can assume that $x_i \leq x_k \leq x_j$ and that $y_i \leq y_j$.
That $p_kp_i \in E(\Gamma(P))$ now follows easily from $|x_i-x_k| \leq |x_j-x_i| \leq e^{\frac12(y_i+y_j)} \leq e^{\frac12(y_k+y_j)}$.

\noindent
{\bf Proof of part~\ref{itm:cross.ii}}
Of course the projections of the two intervals onto the $x$-axis also intersect.
We can assume without loss of generality that $x_k < x_\ell, x_i < x_j$.
(If one of the two segments is vertical then we are done by the previous part.)
If $[x_i, x_j] \subseteq [x_k, x_\ell]$ then either $p_i$ or $p_j$ is above
$[p_k, p_\ell]$ (since the segments cross) and we are done by the previous lemma.
Similarly, we are done if $[x_k, x_\ell] \subseteq [x_i, x_j]$.
Up to symmetry, the remaining case is when $x_i < x_k \leq x_j < x_\ell$.

Suppose now that $y_k \geq y_j$. Then $|x_i-x_k| \leq |x_i-x_j| \leq e^{\frac12(y_i+y_j)} \leq e^{\frac12(y_i+y_k)}$.
In other words, $p_kp_i \in E(\Gamma(P))$. 
In the case when $y_k < y_j$ we find similarly that $p_jp_\ell \in E(\Gamma(P))$.
\end{proofof}

We now consider a Poisson process on the upper half of the (Euclidean) plane $\eR\times[0,\infty)$, with intensity function

\begin{equation}\label{eq:fdef} 
f_{\alpha,\lambda}(x,y) := \lambda \cdot e^{-\alpha y}. 
\end{equation}

\noindent
Here $\alpha, \lambda > 0$ are parameters. 
Let us denote the points of this Poisson process by $\Pcal_{\alpha,\lambda} := \{p_1, p_2, \dots \}$, and let us write
$p_i := (x_i, y_i)$ for all $i$.
We will be interested in the random countable graph $\Gamma_{\alpha,\lambda} := \Gamma(\Pcal_{\alpha,\lambda})$ with 
$\Gamma(.)$ as defined above.

In what follows, we will make frequent use of the following two facts on $\Pcal_{\alpha,\lambda}$.
The following result is a direct consequence of the superposition theorem for Poisson processes 
(which can for instance be found in~\cite{KingmanBoek}). 

\begin{lemma} \label{lem:superpos}
When $\alpha \geq \alpha'$ and $\lambda \leq \lambda'$, then $\Pcal_{\alpha',\lambda'}$ is distributed like
the union of $\Pcal_{\alpha,\lambda}$ and an independent Poisson point process with intensity function
$f_{\alpha',\lambda'} - f_{\alpha,\lambda}$.
\end{lemma}


We will also need to use the following related observation. 

\begin{lemma}\label{lem:couplingPoisson}
The exists a coupling of all the Poisson processes $\Pcal_{\alpha,\lambda}$ with $\alpha, \lambda > 0$ simultaneously such that 
(almost surely) $\Pcal_{\alpha,\lambda} \subseteq \Pcal_{\alpha',\lambda'}$ whenever $\alpha \geq \alpha'$ and $\lambda \leq \lambda'$.
\end{lemma}

\begin{proof}
For completeness we spell out the straightforward proof.
Let $\Qcal$ denote a Poisson process of intensity 1 on $\eR \times (0,\infty)^2$, and let us 
define 

\[ \Qcal_{\alpha, \lambda} := \{ (x,y,z) \in \Qcal : z < \lambda e^{-\alpha y} \}, \quad
 \Pcal_{\alpha,\lambda}' := \pi\left[\Qcal_{\alpha,\lambda}\right], 
\]

\noindent
where $\pi$ denotes the projection on the first two coordinates. 
It is easily checked that $\Pcal_{\alpha,\lambda}'$ is a Poisson process with intensity $f_{\alpha,\lambda}$.
By construction we have the desired inclusions $\Pcal_{\alpha,\lambda}' \subseteq \Pcal_{\alpha',\lambda'}'$ whenever 
$\alpha \geq \alpha'$ and $\lambda \leq \lambda'$. 
\end{proof}

\subsection{Percolation}

%
%
We say that {\em percolation} occurs if the resulting graph $\Gamma_{\alpha,\lambda}$ has an infinite component.

\begin{theorem}\label{lem:perc01} For every $\alpha,\lambda > 0$ it holds that 
$\Pee_{\alpha,\lambda}( \text{percolation} ) \in \{0,1\}$.
\end{theorem}

\begin{proof} Observe that (the probability distribution of) $\Pcal_{\alpha,\lambda}$ is invariant under horizontal translations.
A standard argument (see for instance~\cite{MeesterRoy96}, Proposition 2.6) now shows that all events $E$ that are invariant
under horizontal translations have probability $\Pee(E) \in \{0,1\}$.
\end{proof}

In the sequel, we shall deal with following quantity:
\begin{equation}\label{eq:thetadef} 
\theta( y; \alpha, \lambda ) := \Pee( \text{$\Gamma( \{(0,y)\} \cup \Pcal_{\alpha,\lambda} )$ contains an infinite 
component visiting $(0,y)$} ),
\end{equation}
defined for all $y \geq 0$. 
This is often called the {\em percolation function} in the percolation literature.
The following observations make the link between $\theta$ and the event that percolation occurs more explicit, and will be used in the sequel.

\begin{lemma}\label{lem:linkthetaperc}
We have, for all $\alpha, \lambda > 0$:
\begin{enumerate}
\item\label{itm:thetaposifperc} $\Pee_{\alpha,\lambda}( \text{percolation} ) = 1$ if and only if
$\theta( y; \alpha, \lambda ) > 0$ for all $y \geq 0$;
\item\label{itm:thetazeroifnoperc} $\Pee_{\alpha,\lambda}(\text{percolation} ) = 0$ if and only if  
$\theta(y;\alpha,\lambda) = 0$ for all $y\geq 0$.
\end{enumerate}
\end{lemma}

\begin{proof}
Let us denote $E := \{\text{percolation}\}$ and let $E_y$ denote the event that 
$(0,y)$ is in an infinite component of $\Gamma( \{(0,y)\} \cup \Pcal_{\alpha, \lambda} )$.

Let us first assume that $\Pee(E) = 1$.
For $i \in\Zed$, let us denote%
\[ 
F_i := \left\{ \begin{array}{l} 
	\text{$[i,i+1)\times[0,\infty)$ contains a point of $\Pcal_{\alpha,\lambda}$ that } \\
        \text{ belongs to an infinite component of $\Gamma[\Pcal_{\alpha,\lambda}]$} 
       \end{array} \right\}.
\]

\noindent
We clearly have 

\[ 1 = \Pee(E) = \Pee\left( \bigcup_{i\in\Zed} F_i \right) 
\leq \sum_{i\in \Zed} \Pee( F_i ). \]

\noindent
On the other hand, we must also have that $\Pee(F_i) = \Pee( F_j )$ for all $i,j\in\Zed$
(since the point process $\Pcal_{\alpha,\lambda}$ is invariant under horizontal translations).
It follows that $\Pee( F_0 ) > 0$.
We now remark that $F_0 \subseteq E_y$ for every $y \geq 0$, since if $(x',y') \in [0,1)\times[0,\infty)$ then
$|x'-0| \leq 1 \leq e^{\frac12(y'+y)}$.
This shows that if $\Pee(E) = 1$, then $\Pee(E_y) > 0$ for all $y\geq 0$.

Let us now suppose that $\Pee(E_y) > 0$ for some $y \geq 0$.
Let $G$ denote the event that at least one point of $\Pcal_{\alpha,\lambda}$ falls inside
$[0,1]\times[y+2\ln 2, \infty)$.
By the FKG inequality for Poisson processes (see for instance~\cite{MeesterRoy96}, page 32), we have
$\Pee( E_y \cap G ) \geq \Pee( E_y )\Pee( G ) > 0$.
Now note that if $(x',y') \in \eR\times (0,\infty)$ is such that $|x'| \leq e^{(y+y')/2}$, then we
also have that $|x'| + 1 \leq 2e^{(y+y')/2} = e^{(y+2\ln2+y')/2}$.
This means that, if $G$ holds then there is a point $p \in \Pcal_{\alpha,\lambda}$ that 
is adjacent to every neighbour of $(0,y)$ in $\Gamma[\{(0,y)\}\cup\Pcal_{\alpha,\lambda}]$.
This implies that $E_y \cap G \subseteq E$. Thus, if $\Pee(E_y) > 0$ then also $\Pee( E ) > 0$
and in fact $\Pee(E) = 1$ by Lemma~\ref{lem:perc01}.

The above observations show that $\Pee(E)=1 \Leftrightarrow \Pee(E_y) > 0$ (for every $y$).
This implies both part~\ref{itm:thetaposifperc} and part~\ref{itm:thetazeroifnoperc}.
\end{proof}

In the next lemma, we show that when $\alpha$ crosses the value $1/2$, then $\theta (y; \alpha, \lambda)$ drops below 1. 

\begin{lemma}\label{lem:thetaalphaless12}
We have
\begin{enumerate}
\item\label{itm:thetaalphaless12i} For $\alpha \leq \frac12$, we have that $\theta( y; \alpha, \lambda ) = 1$ for all $y \geq 0$ and $\lambda > 0$;
\item\label{itm:thetaalphaless12ii} For $\alpha > \frac12$, we have that $\theta( y; \alpha, \lambda ) < 1$ for all $y \geq 0$ and $\lambda > 0$.
\end{enumerate}
\end{lemma}

\begin{proof} 
We start with the proof of~\ref{itm:thetaalphaless12ii}.
Let $D$ denote the degree of $(0,y)$ in $\Gamma[ \{(0,y)\}\cup \Pcal_{\alpha,\lambda} ]$.
Then $D$ has a Poisson distribution with mean 

\begin{equation}\label{eq:ED} 
\Ee D = \int_0^\infty \int_{-e^{(y+y')/2}}^{e^{(y+y')/2}} \lambda e^{-\alpha y'}{\dd}x{\dd}y'
 = 2\lambda e^{y/2} \int_0^\infty e^{(\frac12-\alpha)y'}{\dd}y' < \infty.
\end{equation}

(Assuming that $\alpha < 1/2$.)
Hence $(0,y)$ is isolated with probability $e^{-\Ee D} > 0$ and therefore $1-\theta (y;\alpha,\lambda) >0$.

For the proof of~\ref{itm:thetaalphaless12i}, we note that 
the computations~\eqref{eq:ED} give that this time $\Ee D = \infty$.
Hence, by the properties of the Poisson process, almost surely, $(0,y)$ has infinitely many neighbours and
in particular lies in an infinite component.
\end{proof}

We will now restrict our attention to the case $\alpha> 1/2$  and we show that a
nontrivial phase transition in $\lambda$ and $\alpha$ takes place. We begin with the case when $\alpha=1$.

\begin{lemma}\label{lem:perc0}
There exists a $\lambda^- > 0$ such that if $\lambda < \lambda^-$ then 
$\Pee_{1,\lambda}(\text{percolation}) = 0$.
\end{lemma}

\begin{proof}
We will show that, provided that $\lambda$ is small enough, we have $\theta (0;1,\lambda) = 0$. 
By Lemma~\ref{lem:linkthetaperc} this will prove the result. 

Let us thus consider the component of $(0,0)$ in $\Gamma[\{(0,0)\} \cup \Pcal_{1,\lambda}]$.
We will iteratively construct a sequence of points $(x_0, y_0), (x_1,y_1), \dots$. The sequence may be either infinite or finite
and it will have the following properties: 
\begin{itemize}
\item[{\bf(p-i)}] $x_0=y_0=0$;
\item[{\bf(p-ii)}] $x_i < x_{i+1}$;
\item[{\bf(p-iii)}] If the sequence is finite, then the component of $(0,0)$ is 
contained in $[-\max_i x_i, \max_i x_i ] \times [0,\infty)$.
\end{itemize}
For notational convenience we will write $p_i^+ = (x_i, y_i)$ and $p_i^- = (-x_i, y_i)$.

Assume that $x_i$ and $y_i$ have been determined, for some $i\geq 0$. We consider the points $(\Bcal^+ (p_i^+) \cup \Bcal^-(p_i^-) ) \cap 
\Pcal_{1,\lambda}$, that is, the points of $\Pcal_{1,\lambda}$ that either belong to the right half-ball around $p_i^+$
or the left half-ball around $p_i^-$. 
If there are no such points, then the construction stops (and $x_j, y_j$ are not defined for $j \geq i+1$).
Otherwise, we set

\[ \begin{array}{l}
x_{i+1} := \max \{ |x| \ : \ p'=(x,y) \in (\Bcal^+ (p_i^+) \cup \Bcal^-(p_i^-) )  \cap \Pcal_{1,\lambda}\}, \\ 
y_{i+1} := \max \{ y \ : \ p'=(x,y) \in (\Bcal^+ (p_i^+) \cup \Bcal^-(p_i^-) )  \cap \Pcal_{1,\lambda}\}. \\
   \end{array} \]

\noindent
The conditions {\bf(p-i)} and {\bf(p-ii)} are clearly met by the construction.
Let us remark that the points $p_i^+$ and $p_i^-$ may not belong to $\Pcal_{1,\lambda}$ and that, by construction, there is no point 
among $\Bcal^+ (p_i^+) \cap \Pcal_{1,\lambda}$ that is higher than $p_{i+1}^+$ or to the right of it (and similarly for $p_i^-$). 


Now suppose that $(x_i,y_i)$ is defined and there exist points $p' \in \Pcal_{1,\lambda} \setminus [-x_i,x_i]\times[0,\infty)$ and 
$p'' \in \Pcal_{1,\lambda} \cap [-x_i,x_i]\times[0,\infty)$ such that $p' \in \Bcal(p'')$.
We will show that in this case, the sequence does not terminate in step $i$, i.e.,~$(x_{i+1},y_{i+1})$ will be defined as well.
This will prove our construction that satisfies condition {\bf(p-iii)} above, since the sequence can only stop at step $i$
if the component of $(0,0)$ is contained in $[-x_i,x_i]\times[0,\infty)$.

By symmetry, we can assume without loss of generality that $x(p') > x_i$.
Let $1 \leq j \leq i$ be such that $x_{j-1} < |x(p'')| \leq x_j$. (Note that a.s. there is no point in 
$\Pcal_{\alpha, \lambda}$  with $x$-coordinate equal to 0, so $j$ is well-defined a.s.) Note that $y(p'') \leq y_j$ by the definition of $y_j$. 
We see that 

\[ |x_j - x(p')| \leq |x(p'')-x(p')| \leq e^{\frac12(y(p'')+y(p'))} \leq e^{\frac12(y_j+y(p'))}. \]

\noindent
In other words, $p' \in \bigcup_{j=0}^i \Bcal(p_j^+)$. 
Since $x(p')>x_i$ we must have
$p' \in \Bcal_+(p_i^+)$ -- otherwise $x_i$ would be even bigger by definition of of the sequence $(x_n,y_n)$. 
Thus, $(x_{i+1}, y_{i+1})$ is defined so that condition {\bf(p-iii)} is indeed met.

Next, we will show that, provided that $\lambda$ is sufficiently small, almost surely the sequence $(x_0,y_0), (x_1,y_1), \dots$ 
is finite. That is, at some point in the construction, the set $\Bcal^-(p_i^-) \cup \Bcal^+(p_i^+)$ does not contain any point of $\Pcal_{1,\lambda}$.

Suppose that we have already found $(x_1,y_1), \dots, (x_i,y_i)$. 
Observe that, by construction, both $\Bcal^+(p_i^+) \cap \bigcup_{j=0}^{i-1} \Bcal^+(p_j^+)$ 
and $\Bcal^-(p_i^-) \cap \bigcup_{j=0}^{i-1} \Bcal^-(p_j^-)$
cannot contain any point of $\Pcal_{1,\lambda}$ (for otherwise the value of $x_i$ would be even larger). 
Thus, $y_{i+1}$ is the maximum $y$-coordinate
among points of $\Pcal_{1,\lambda}$ inside $\left(\Bcal^+(p_i^+) \setminus \bigcup_{j=0}^{i-1} \Bcal^+(p_j^+)\right)
\cup \left(\Bcal^-(p_i^-) \setminus \bigcup_{j=0}^{i-1} \Bcal^-(p_j^-)\right)$ 
(provided there is at least one such point -- for convenience let us set $y_{i+1} = -\infty$ otherwise).
The expected number $\mu(z|x_1,y_1,\dots,x_i,y_i)$ of points in $\left(\Bcal^+(p_i^+) \setminus \bigcup_{j=0}^{i-1} \Bcal^+(p_j^+)\right)
\cup \left(\Bcal^-(p_i^-) \setminus \bigcup_{j=0}^{i-1} \Bcal^-(p_j^-)\right)$  with $y$-coordinate at least $z$ satisfies:

\[ \begin{array}{rcl} 
    \mu(z|x_1,y_1,\dots,x_i,y_i)
    & \leq & 
    \displaystyle 2\int_z^\infty \int_{x_i}^{x_i+t(y',y_i)} f_{1,\lambda} (x',y') dx' dy'  \\
    & = & 
    \displaystyle 2\lambda e^{\frac{y_i}{2}}\int_z^{\infty} e^{-\frac{y'}{2}}dy' \\
    & = &   
    4\lambda e^{(y_i-z)/2}. 
   \end{array}
\]

\noindent
Thus 

\begin{equation}\label{eq:gumbel} 
\Pee( y_{i+1} \leq z | x_1,y_1,\dots,x_i,y_i ) = e^{-\mu(z|x_1,y_1,\dots,x_i,y_i)} \geq e^{-4\lambda e^{(y_i-z)/2}}. 
\end{equation}

\noindent
Writing $\Delta_i := y_{i}-y_{i-1}$, we see that these increments $\Delta_i$ are stochastically dominated
by an i.i.d.~sequence $\{ \hat{\Delta}_i\}_{i\in \mathbb{N}}$ with common cumulative distribution function 

\[ F_{\hat{\Delta}_1}(x) = e^{-4\lambda e^{-x/2}} = e^{-e^{-\frac{(x-2\ln (4\lambda) )}{2}}}. \]

\noindent
This is a Gumbel distribution, and it satisfies $\Ee \hat{\Delta}_1 = 2\ln (4\lambda) + 2\gamma$, where $\gamma$ is Euler's 
constant, and $\Var\hat{\Delta}_1 = 2\pi^2/3$ (see for example~\cite{Stats} p. 542). 
Observe that for $\lambda < e^{-\gamma}/4$ we have $\Ee \hat{\Delta}_1 < 0$. Let us thus fix such a $\lambda$.
We have that

\[ \Pee( \text{$(x_i,y_i)$ exists} ) \leq \Pee( \hat{\Delta}_1+\dots+\hat{\Delta}_i \geq 0 ). \]

\noindent
The law of large numbers now implies that the right hand side tends to zero as $i\to\infty$.
It follows that the sequence $(x_0,y_0), (x_1,y_1), \dots$ is finite with probability one. 
Thus, with probability one, there is a $k \in \eN$ such that the component of $(0,0)$ is contained in the strip
$S_k := [-k,k]\times [0,\infty)$.
Let us remark that $|S_k \cap \Pcal_{1,\lambda}|$ is a Poisson random variable with mean $2k\lambda < \infty$.
Hence, almost surely, $|S_k \cap \Pcal_{1,\lambda}|$ is finite for all $k$. So in particular the component of $(0,0)$ is also almost surely
finite.
This proves that $\theta(0;1,\lambda) = 0$ when $\lambda < e^{-\gamma}/2$.
\end{proof}

The above proof can be extended and show the fact that when $\alpha > 1$, then for any $\lambda >0$ a.s. 
no percolation occurs. However, we shall not need to do this here as our goal is to study the largest component of
the KPKVB model, and the corresponding case has already been covered in~\cite{BFMgiantEJC}. 

Next, we will show that for $\alpha = 1$, there exists $\lambda^+>0$ such that 
when $\lambda > \lambda^+$, then percolation occurs with positive probability. 
As we shall see later (cf. Lemma~\ref{lem:lambdacexists}), this fact together with Lemma~\ref{lem:perc0} imply the existence 
of a critical value for the parameter $\lambda$ around which an infinite component emerges. 
Furthermore, we also prove that for any $\alpha <1$ and any $\lambda >0$ percolation occurs with positive probability.

The proofs of these facts follow the same strategy. 
More specifically, we will consider a discretisation of the model, by dividing the upper half-plane $\eR \times [0,\infty)$ into 
rectangles that contain in expectation the same number of points from $\Pcal_{1,\lambda}$. Furthermore, the rectangles are such 
that for any two rectangles that have intersecting sides any two points that are contained in them are adjacent in $\Gamma_{1,\lambda}$.
We define rectangles 

\begin{equation}\label{def:Rij}
 R_{i,j} := \{ (x,y) : i\ln 2 < y \leq (i+1)\ln 2,  j 2^{i-1} < x \leq (j+1)2^{i-1}\},
\end{equation}
 \noindent
where $i \geq 0$ and $j \in \mathbb{Z}$. 
See Figure~\ref{fig:Rij} for a depiction.
\begin{figure}[h]
\begin{center}
\input{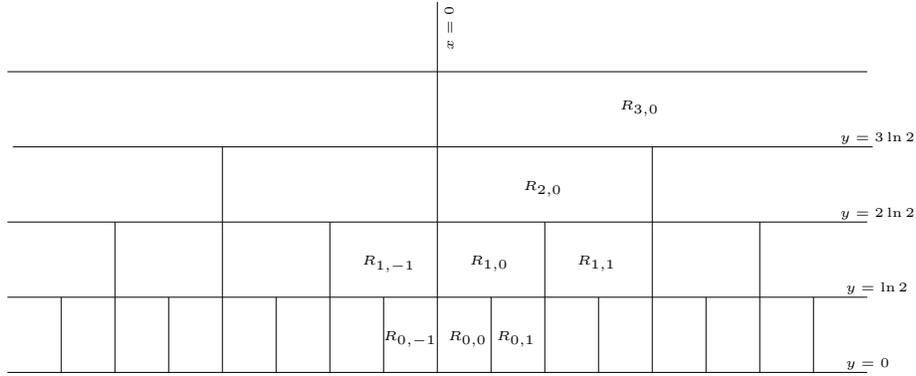}
\end{center}
\caption{The dissection into the rectangles $R_{i,j}$.\label{fig:Rij}}
\end{figure}
Observe that for $i\geq 1$, the rectangle $R_{i,j}$ shares an edge (from below) with the rectangles $R_{i-1,2j}$ and $R_{i-1,2j+1}$ 
and (from above) with the rectangle $R_{i+1,\lfloor j/2 \rfloor}$. It also shares its side edges with the rectangles $R_{i,j-1}$ and 
$R_{i,j+1}$. We assert that:

\begin{lemma} \label{eq:Rij}
For all $p \in R_{i,j}$ we have 
\begin{equation*} 
\Bcal (p) \supseteq R_{i+1,\lfloor j/2 \rfloor} \cup R_{i,j-1} \cup R_{i,j+1} \cup R_{i-1,2j} \cup R_{i-1,2j+1}. 
\end{equation*}
\end{lemma}
\begin{proof}
Indeed, let $p' \in R_{i,j+1}$. Then 

\[ |x(p)-x(p')| \leq 2\cdot 2^{i-1} = e^{\frac{2i\ln 2}{2}} < e^{\frac{y(p)+ y(p')}{2}}. \]

\noindent 
By symmetry, the same holds if $p' \in R_{i,j-1}$. 
If $p' \in R_{i+1,\lfloor j/2 \rfloor}$, then 

\[ |x(p) - x(p')|\leq 2\cdot 2^{i-1} < 2^{i+1} = e^{\frac{ i \ln 2 + (i+2)\ln 2 }{2}} < e^{\frac{y(p) + y(p')}{2}}. \]
\noindent
Note that this also implies that for any $p \in R_{i,j}$ we have $\Bcal (p) \supset R_{i-1,2j}, R_{i-1,2j+1}$. 
This concludes the proof of the Lemma.
\end{proof}

\noindent
Phrased differently, Lemma~\ref{eq:Rij} states that any two points of $\Pcal_{\alpha,\lambda}$ that belong to adjacent boxes must 
be adjacent in $\Gamma_{\alpha, \lambda}$.   

The general strategy here is to consider a graph (which we will denote by $\Rscr$) whose vertex set consists of those boxes 
that contain at least one point and any pair of such boxes are adjacent in this graph, if they touch along a side (i.e.~they share more than just a corner). 
The above lemma implies that if the graph $\Rscr$ contains an infinite component, then 
$\Gamma_{\alpha, \lambda}$ percolates. 

The following gives the expected number of points of $\Pcal_{\alpha,\lambda}$ that fall inside $R_{i,j}$.
\begin{equation}\label{eq:RijExpect}
\Ee |R_{i,j} \cap \Pcal_{\alpha,\lambda}| = \lambda \int_{i\ln 2}^{(i+1)\ln 2} \int_{0}^{2^{i-1}} 
e^{-\alpha y}  dxdy = \frac{\lambda}{\alpha} 2^{i-1} \left( 2^{-\alpha i} - 2^{-\alpha(i+1)} \right).
\end{equation}
We will use this formula below. 

The next lemma uses this discretisation in order to show that when $\alpha =1$ and $\lambda$ is large enough, then percolation 
occurs with probability 1. 

\begin{lemma}\label{lem:perc1} 
There exists a $\lambda^+ < \infty$ such that if $\lambda > \lambda^+$ then 
$\Pee_{1,\lambda}(\text{percolation}) = 1$.
\end{lemma}

\begin{proof} 
We consider the above discretisation and declare \emph{active} a rectangle that contains at least one point from $\Pcal_{1,\lambda}$. 
We define the graph $\Rscr$ whose vertex set is the (random) set of active rectangles and a pair of active rectangles share an edge 
if they touch along a side. 
The definition of the discretisation and, in particular, Lemma~\ref{eq:Rij} implies that if the graph $\Rscr$ contains an infinite connected component then $\Gamma_{\alpha, \lambda}$ contains one as well. 

We will show that the probability that a rectangle is active can become as close to 1 as we please, if we make 
$\lambda$ large enough.  
In the light of Theorem~\ref{lem:perc01} it thus suffices to show that, for $\lambda$ sufficiently large, 
the rectangle $R_{0,0}$ is in an infinite component of $A_{1,\lambda}$ with
positive probability.
Hence, to bound the probability that a point $(0,y)$ belongs to an infinite component is greater than 0, it suffices
to show that the box wherein it is located belongs to an infinite component in $\Rscr$ with positive probability, provided 
that $\lambda$ is large enough. Once we have shown this, the lemma will follow from Theorem~\ref{lem:perc01}.  

By (\ref{eq:RijExpect}),  for any $i\geq 0$ and $j \in \mathbb{Z}$, the expected number of points of $\Pcal_{1,\lambda}$ in $R_{i,j}$ is 
\begin{equation*} 
\Ee |R_{i,j}\cap\Pcal_{1,\lambda}| = \lambda 2^{i-1} \left( 2^{-i} - 2^{-i-1} \right) = \frac{\lambda}{4}. 
\end{equation*}

\noindent
Therefore, the probability that the rectangle is active is $p_{act}:=1-e^{-\frac{\lambda}{4}}$. We set $q_{act} := 1 - p_{act} = 
e^{-\frac{\lambda}{4}}$.

Let $\overline{\Rscr}$ (in some sense the complement of $\Rscr$) denote the graph whose vertex set is the set of 
\emph{inactive} rectangles where any two such rectangles are adjacent in this graph if they touch (either along a side or just in a 
a corner). 
Observe that if $R_{0,0}$ is not in an infinite component of $\Rscr$ then either 
$R_{0,0}$ is not active, or there is a path in $\overline{\Rscr}$ between a rectangle $R_{0,j^-}$ and a rectangle $R_{0,j^+}$ with 
$j^- < 0 < j^+$. This path must pass through some rectangle $R_{i,0}$ with $i > 0$.
Simplifying matters even further, we can say that if $R_{0,0}$ is not in an infinite component of $\Rscr$ then either 
$R_{0,0}$ is not active or, for some $i>0$ there is a path of length $i$ starting in $R_{i,0}$.
Since each rectangle touches at most 8 other rectangles, we see that 

\begin{equation}\label{eq:qact} 
\Pee( \text{component of $R_{0,0}$ is finite} ) 
 \leq q_{act} + \sum_{i=1}^\infty (7q_{act})^i = q_{act} \cdot \left(1 + \frac{7}{1-7q_{act}}\right). 
 \end{equation}

 \noindent
It is clear that by choosing $\lambda$ sufficiently large, we can make $q_{act}$ as small as we like.
And, it is also clear that for $q_{act}$ sufficiently small, the right hand side of~\eqref{eq:qact} is smaller than 1.
Thus, for sufficiently large $\lambda$, the probability that there is an infinite component in $\Gamma_{1,\lambda}$ is positive
(and hence equals one).
\end{proof}

\noindent
At this point, we can deduce the existence of a critical $\lambda$, when $\alpha =1$.
\begin{lemma}\label{lem:lambdacexists}
There exists a $\lambdac > 0$ such that

\[ \Pee_{1,\lambda}(\text{percolation} ) = \left\{\begin{array}{cl}
                                                   0 & \text{ if $\lambda < \lambdac$, } \\
                                                   1 & \text{ if $\lambda > \lambdac$. }
                                                  \end{array} \right. \]
\end{lemma}

\begin{proof}
 This follows immediately from Theorem~\ref{lem:perc01} and Lemmas~\ref{lem:perc0},~\ref{lem:perc1} together
with the observation that $\Pee_{1,\lambda}(\text{percolation})$ is nondecreasing in $\lambda$ by Lemma~\ref{lem:couplingPoisson}.
\end{proof}

\noindent
The next lemma shows that percolation occurs always when $\alpha < 1$, 
independently of the value of $\lambda$. The proof is an easy adaptation of the proof of Lemma~\ref{lem:perc1}.

\begin{lemma}\label{lem:perc3}
For every $\alpha < 1$ and $\lambda > 0$ we have
$\Pee_{\alpha,\lambda}(\text{percolation} ) = 1$.
\end{lemma}

\begin{proof}
We again consider the graph $\Rscr$ defined in the proof of Lemma~\ref{lem:perc1}.
Using (\ref{eq:RijExpect}), the expected number of points of $\Pcal_{\alpha, \lambda}$ that fall inside the rectangle
$R_{i,j}$ is

\begin{equation}\label{eq:mijdef} 
\mathbb{E} |R_{i,j} \cap \Pcal_{\alpha, \lambda}| = \frac{\lambda}{2\alpha}(1-2^{-\alpha})2^{i(1-\alpha)}.
\end{equation}

\noindent
Note that this does not depend on $j$ and tends to infinity with $i$.
Thus, for every $\eps > 0$ there is an $i_0 = i_0(\eps,\lambda,\alpha)$ such that 
$\Pee( R_{i,j} \cap \Pcal_{\alpha,\lambda} = \emptyset ) =  e^{-\mathbb{E} |R_{i,j} \cap \Pcal_{\alpha}|} < \eps$ for all $i \geq i_0$ and 
all $j \in \Zed$.

We fix a sufficiently small $\eps$, to be made precise later, and we consider the corresponding $i_0$.
Arguing as in the proof of Lemma~\ref{lem:perc1}, we observe that if $R_{i_0,0}$ is not in an infinite component of $\Rscr$ then either
$R_{i_0,0}$ is not active or for some $j > 0$ there is a path of length $j$ in $\overline{\Rscr}$ starting in $R_{i_0+j,0}$.
Analogously to~\eqref{eq:qact} we find

\[ \Pee( \text{$R_{i_0,0}$ in a finite component} ) \leq \eps \left( 1 + \frac{7}{1-7\eps}\right) < 1, \]

\noindent
where the last inequality holds provided $\eps$ was chosen sufficiently small.
This proves the lemma.
\end{proof}

%

Having identified the range of $\alpha$ and $\lambda$ where percolation occurs, we proceed with showing that if there is an 
infinite component, then it is a.s. unique. 

\begin{lemma}\label{lem:therecanbeonlyone}
For every $0 < \alpha \leq 1$ and $\lambda > 0$, almost surely, there is at most one infinite component.
\end{lemma}

\begin{proof} 
We consider the dissection of $\eR \times [0,\infty)$ that was defined in~\eqref{def:Rij}. 
Consider the boxes $R_{i_1,j_1}$ and $R_{i_2,j_2}$. Let $i_0$ be the smallest $i$ such that 
$2^{i-1} > |j_1| 2^{i_1-1}, |j_2|2^{i_2-1}$. 
In other words, this is the smallest $i$ for which $R_{i,-1}$ and $R_{i,1}$ both lie above the two boxes. 
Note that $i_0 > \max \{ i_1,i_2 \}$.

As we have seen in the proof of Lemma~\ref{lem:perc3}, cf.~\eqref{eq:mijdef}, each box is active (that is, it contains at least 
one vertex) with probability at least $1 - e^{-\frac{\lambda}{2\alpha}(1-2^{-\alpha})} =: p>0$. 
So for any $i >0$, the rectangles $R_{i,-2}, R_{i,-1}, R_{i,1},R_{i,2}$ are all active with 
probability at least $p^4$. Let $E_i$ denote this event. 
If $E_i$ is realised, then there are vertices 
$u_0 \in R_{i,-2}, u_1 \in R_{i,-1}, v_1 \in R_{i,1}$ and $v_0 \in R_{i,2}$ and by the definition of the boxes these 
four vertices form the path $u_0u_1v_1v_0$. 

Now, if $i\geq i_0$, then these vertices lie above the boxes $R_{i_1,j_1}, R_{i_2,j_2}$. Furthermore, with probability 1 there is 
an $i\geq i_0$ such that $E_i$ is realised. 

\begin{figure}[ht] 
\begin{center}
\includegraphics[scale=0.5]{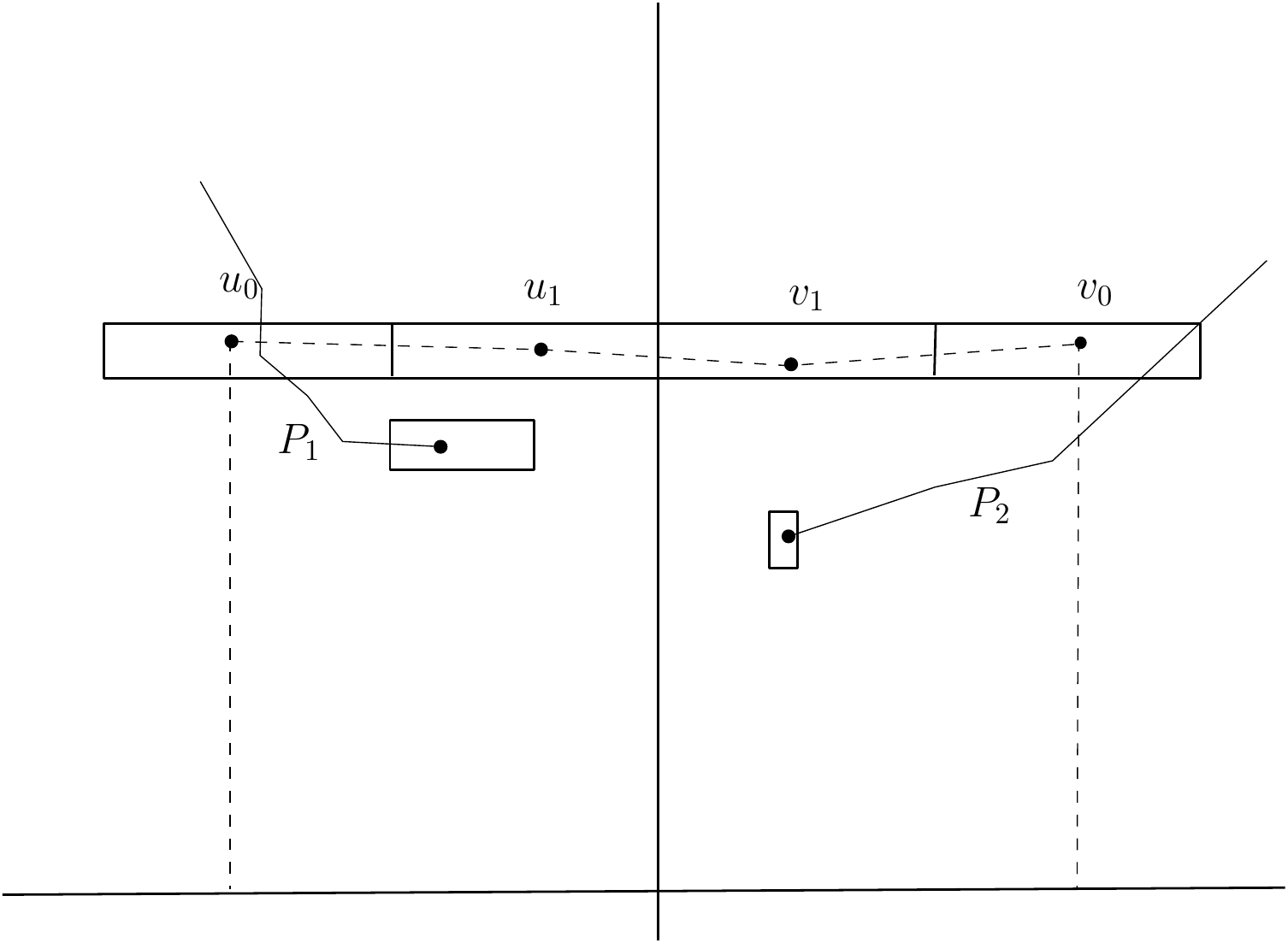}
\end{center}
\caption{The event $R_i$.\label{fig:Unique}} 
\end{figure} 

We claim that if an infinite path $P_1$ visits box $R_{i_1,j_1}$ and an infinite path $P_2$ visits box $R_{i_2,j_2}$, then 
they belong to same connected component. Indeed, consider the region defined by the segments 
$[u_0,u_1], [u_1,v_1], [v_1,v_0]$, the lines $x=x(u_0), x=x(v_0)$ and the $x$-axis - let $\mathcal{R}$ denote it. 
(See Figure~\ref{fig:Unique} for a depiction.)
Let $P_1$ be an infinite path that visits box $R_{i_1,j_1}$, and consider an edge $e$ of $P_1$ 
that has one end inside $\mathcal{R}$ and one point outside. 
Observe that one of the following three options has to be the case : {\bf 1)} $e$ crosses at least one of the edges $u_0u_1, u_1v_1, v_1v_0$, or 
{\bf 2)} $e$ crosses the line $x=x(u_0)$ below $u_0$, or {\bf 3)} $e$ crosses the line $x=x(v_0)$ below $v_0$.
In all three cases Lemma~\ref{lem:cross} implies that $P_1$ is in the same component as the path $u_0u_1v_1v_0$. 
Completely analogously, $P_2$ is also in the same component as the path $u_0u_1v_1v_0$.

Thereby, we can conclude that $P_1$ and $P_2$ belong to the same component. 
\end{proof}

%
%
%
%

\section{The function $c (\alpha, \nu)$}
The function $c (\alpha,\nu )$ which appears in Theorem~\ref{thm:main} is defined as follows. 
\begin{equation}\label{eq:cdef} 
 c(\alpha, \nu) := \left\{\begin{array}{cl} 
 \int_0^\infty \theta\left(y;\alpha,\nu\alpha/\pi\right) \alpha e^{-\alpha y} {\dd}y & \text{ if } \alpha \leq 1, \\
  0 & \text{ if } \alpha > 1.
                         \end{array} \right.
\end{equation}
Property~\ref{itm:main.i} holds by definition of $c$. Property~\ref{itm:main.ii} follows immediately from Lemmas~\ref{lem:perc3}
and~\ref{lem:linkthetaperc}.
Property~\ref{itm:main.iii} follows immediately from Lemmas~\ref{lem:lambdacexists} and~\ref{lem:linkthetaperc}, where of course
$\nuc = \pi \lambdac$.
Property~\ref{itm:main.iv} follows immediately from Lemma~\ref{lem:thetaalphaless12}.

To deduce parts {\bf (vi)} and {\bf (v)} of Theorem~\ref{thm:main}, we will first deduce the monotonicity and continuity properties of 
$\theta (y; \alpha, \lambda)$. We will begin with the former, as they are slightly easier than the latter. 

\subsection{The monotonicity of $\theta (y;\alpha, \lambda)$} 
Property~\ref{itm:main.vi} follows immediately from the next lemma.

\begin{lemma}\label{lem:thetadecralpha} We have
\begin{enumerate}
\item\label{itm:thetadecralpha} If $\frac12<\alpha<1$ and $\alpha' < \alpha$ then $\theta(y;\alpha,\lambda) < \theta(y;\alpha',\lambda)$ for all $y,\lambda$.
\item\label{itm:thetaincrlambda} If $\frac12<\alpha<1$ and $\lambda > 0$ or if $\alpha=1$ and $\lambda > \lambdac$ then
$\theta(y;\alpha,\lambda) < \theta(y;\alpha,\lambda')$ for all $y$ and all $\lambda'>\lambda$.
\end{enumerate}
\end{lemma}

\begin{proof}
We start with~\ref{itm:thetadecralpha}.
Using lemma~\ref{lem:superpos}, we write $\Pcal_{\alpha',\lambda}$ as the union of
$\Pcal_{\alpha,\lambda}$ and an independent Poisson point process $\Pcal'$ with intensity $g := f_{\alpha',\lambda}-f_{\alpha,\lambda}$.

Let $E$ denote the event that $(0,y)$ is in an infinite component of $\Gamma[ \{(0,y)\} \cup \Pcal_{\alpha,\lambda} ]$.
By Lemma~\ref{lem:thetaalphaless12} we have $\Pee(E) < 1$. 

If for $h>0$ we denote by $F_h$ the event that the box $[-h,h]\times [0,h]$ contains a vertex in an infinite component of 
$\Gamma_{\alpha,\lambda}$, we then have  $\Pee(\cup_{h>0} F_h) = 1$ by Lemma~\ref{lem:perc3}.  
Hence, there is an $h > y$ such that $\Pee( E^c \cap F_h )>0$. 

Let $G_h$ denote the event that there is a point of $\Pcal'$ inside $[0,1]\times[2\log(h+1), \infty)$.
Observe that $G_h$ is independent of $E, F_h$ and that $\Pee( G_h ) > 0$ as $g$ is positive on $[0,1]\times[2\log(h+1), \infty)$.
It follows that 

\[ \Pee( E^c \cap F_h \cap G_h ) > 0. \]

\noindent
Next, observe that if $F_h \cap G_h$ holds, then $(0,y)$ is in an infinite component of 
$\Gamma[ \{(0,y)\} \cup \Pcal_{\alpha',\lambda} ]$.
This gives

\[ \theta( y; \alpha',\lambda ) \geq \Pee( E ) + \Pee( E^c \cap F_h \cap G_h ) > \Pee(E) = \theta( y; \alpha,\lambda ), \]

\noindent
which proves part~\ref{itm:thetadecralpha}.
The proof of~\ref{itm:thetaincrlambda} is completely analogous.
\end{proof}

\subsection{The (dis-)continuity of $\theta (y; \alpha, \lambda)$} 

In this sub-section, we give a collection of results towards the proof of Theorem~\ref{thm:main} (v). 
We will show that the percolation probability $\theta$ is continuous with respect to $\alpha, \lambda, y$
for many choices of these parameters.  
We begin with continuity in $y$.

\begin{lemma}\label{lem:thetactsy}
$\theta(y;\alpha, \lambda)$ is continuous in $y$.
\end{lemma}

\begin{proof}
If $\alpha \leq \frac12$, then there is nothing to prove by Lemma~\ref{lem:thetaalphaless12}. 
Let us thus suppose that $\alpha > \frac12$, and let $y > y' > 0$ be arbitrary. 
Since every point of $\Pcal_{\alpha,\lambda}$ that is adjacent to $(0,y')$ will also 
be adjacent to $(0,y)$, it is clear that $\theta(y;\alpha,\lambda) \geq \theta(y';\alpha,\lambda)$.
Moreover, we have that%
\[ \theta(y;\alpha,\lambda) - \theta(y';\alpha, \lambda) \leq 
 \Pee(\text{There is a point of $\Pcal_{\alpha,\lambda}$ that is adjacent to
 $(0,y)$ but not to $(0,y')$} ) 
\]

Let $\mu$ denote the expected number of points of
$\Pcal_{\alpha,\lambda}$ that are adjacent to $(0,y)$ but not to $(0,y')$.
We have that

\[ 0 \leq \theta(y;\alpha,\lambda) - \theta(y';\alpha,\lambda) \leq \mu \]

\noindent
Let us now compute that

\[ \mu = \int_0^\infty 2\left(e^{\frac12(t+y)}-e^{\frac12(t+y')}\right) \cdot \lambda e^{-\alpha t} {\dd}t
 = 2\lambda \left(e^{y/2}-e^{y'/2}\right) / (\alpha - \frac12).
\]

\noindent
Thus $\mu$ can be made arbitrarily small by choosing $y,y'$ sufficiently close to each other.
It follows that $\theta$ is continuous in its first argument as claimed.
\end{proof}

In the next lemma, we show that the probabilities of certain events under the measure $\Pee_{\alpha, \lambda}$ are continuous with respect to these
two parameters. We will use this several times later on.
\begin{lemma}\label{lem:VerwBegrCont}
Let $E$ be an event that depends only on the points inside a measurable set $A \subseteq \eR \times [0,\infty)$, and 
suppose that $\alpha_0 > 0$ is such that $\int_A e^{-\alpha_0 y}{\dd}x{\dd}y < \infty$.
Then $(\alpha,\lambda) \mapsto \Pee_{\alpha,\lambda}(E)$ is continuous on $(\alpha_0,\infty)\times(0,\infty)$.
\end{lemma}

\begin{proof} 
We start with the continuity in $\lambda$.
To this end, we pick $\alpha > \alpha_0$ and $0 < \lambda' < \lambda$.
By Lemma~\ref{lem:superpos}, we can couple $\Pcal_{\alpha,\lambda}, \Pcal_{\alpha,\lambda'}$ and $\Pcal_{\alpha,\lambda-\lambda'}$ such that
$\Pcal_{\alpha,\lambda}$ is the superposition of $\Pcal_{\alpha,\lambda'}$ and an independent copy of $\Pcal_{\alpha,\lambda-\lambda'}$.
Thus, 

\[ \begin{array}{rcl}
|\Pee_{\alpha,\lambda}(E) - \Pee_{\alpha,\lambda'}(E)| 
& \leq &
 \Pee(\text{there exists a point of $\Pcal_{\alpha,\lambda-\lambda'}$ that falls in $A$} ) \\
& \leq &
|\lambda-\lambda'| \int_A e^{-\alpha y}{\dd}x{\dd}y.
\end{array} \]

\noindent
Since $\int_A e^{-\alpha y}{\dd}x{\dd}y < \infty$ we can thus make the left-hand side arbitrarily small by taking $\lambda$ and 
$\lambda'$ close enough.

Continuity in $\alpha$ is similar.
Pick $\alpha > \alpha' > \alpha_0$.
By lemma~\ref{lem:superpos} we can couple $\Pcal_{\alpha,\lambda}, \Pcal_{\alpha',\lambda}$ so that 
$\Pcal_{\alpha',\lambda}$ is the superposition of $\Pcal_{\alpha,\lambda}$ with an independent Poisson process
with density $g_{\alpha,\alpha',\lambda}(x,y) := \lambda (e^{-\alpha'y} - e^{-\alpha y})$.
Reasoning as before, we see that

\[ 
|\Pee_{\alpha,\lambda}(E) - \Pee_{\alpha,\lambda'}(E)| \leq 
\int_A g_{\alpha,\alpha',\lambda}(x,y) {\dd}x{\dd}y.
\]

\noindent
Since $0 \leq g_{\alpha,\alpha',\lambda}(x,y) \leq \lambda e^{-\alpha_0 y}$, it follows from the dominated convergence theorem
that we can make the left-hand side arbitrarily small by choosing $\alpha, \alpha'$ close.
\end{proof}

The next lemma proves the continuity of $\theta (y;\alpha, \lambda)$ from above with respect to $\lambda$ and from below with respect to $\alpha$. 

\begin{lemma}\label{lem:thetactsabove}
For every $y \geq 0$ and $\alpha, \lambda > 0$ we have that%
\[ \theta(y;\alpha, \lambda) = \lim_{\lambda'\downarrow\lambda} \theta(y;\alpha,\lambda') =  
\lim_{\alpha'\uparrow\alpha} \theta(y;\alpha',\lambda). \]

\end{lemma}

\begin{proof}
The result is clearly trivial when $\alpha \leq \frac12$, by Lemma~\ref{lem:thetaalphaless12}.
Hence we can assume $\alpha > \frac12$.
Let us first remark that $\theta(y;\alpha,\lambda)$ is non-decreasing in $\lambda$. This follows from Lemma~\ref{lem:superpos}.

Let us fix $y,\lambda$.
For any $K>0$ we define the event
\[ E_K := \{ \text{there is a path 
 in $\Gamma (\Pcal_{\alpha,\lambda} \cup \{(0,y)\})$ starting at 
 $(0,y)$ that exits the box $[-K,K]\times[0,K]$}\}.
\]

\noindent
We have that $E_K \supseteq E_{K+1}$ for all $K$, and

\[ \theta(y; \alpha,\lambda) = \Pee_{\alpha,\lambda}( \bigcap_{K>y} E_K )
 = \lim_{K\to\infty} \Pee_{\alpha,\lambda}( E_K ).
\]

\noindent
Hence, for $\eps > 0$ arbitrary, we can find a $K$ such that 
$\Pee_{\alpha,\lambda}(E_K) \leq \theta(y;\alpha,\lambda) + \eps/2$.
Now notice that the event $E_K$ depends only on the points of $\Pcal_{\alpha,\lambda}$
in the set 

\[ \begin{array}{rcl} 
A_K 
& := &
 \{ (x',y') : \exists (x'',y'') \in [-K,K]\times[0,K] \text{ such that } (x',y')\in \Bcal ((x'',y'')) \} \\
& = & 
\{ (x',y') : |x'| \leq K+e^{\frac12(y'+K)} \}.
\end{array} \]

\noindent
(Every point that is connected by an edge to a point in $[-K,K]\times[0,K]$ must lie in 
$A_K$.)
We pick an arbitrary $\alpha > \alpha_0 > \frac12$ and compute 

\[ \int_{A_K} e^{-\alpha_0 t}{\dd}s{\dd}t
 = \int_0^\infty (K+e^{\frac12(t+K)})e^{-\alpha_0 t}{\dd}t 
 = (K + 2e^{K/2})/(\alpha_0-\frac12) < \infty.
\]

\noindent
Hence, Lemma~\ref{lem:VerwBegrCont} applies. In particular, there exists a $\delta > 0$ such 
that $\lambda < \lambda' < \lambda+\delta$ implies that 
$\Pee_{\alpha,\lambda'}(E_K) \leq \Pee_{\alpha,\lambda}(E_K) + \eps/2$.
Hence, for such $\lambda'$, we have

\[ \theta(y;\alpha,\lambda) \leq \theta(y;\alpha,\lambda') \leq
 \Pee_{\alpha,\lambda'}(E_K) \leq \theta(y;\alpha,\lambda) + \eps.
\]

\noindent
As $\eps$ was arbitrary, we indeed see that $\displaystyle \theta(y;\alpha,\lambda) = \lim_{\lambda'\downarrow\lambda} \theta(y;\alpha,\lambda')$.

Completely analogously, there is a $\delta > 0$ such that $\alpha > \alpha' > \alpha-\delta$ implies that
$\Pee_{\alpha',\lambda}(E_K) \leq \Pee_{\alpha,\lambda}(E_K) + \eps/2$ and hence 

\[ \theta(y;\alpha,\lambda) \leq \theta(y;\alpha',\lambda) \leq
 \Pee_{\alpha',\lambda}(E_K) \leq \theta(y;\alpha,\lambda) + \eps.
\]

\noindent
As $\eps$ was arbitrary, we can again conclude that $\displaystyle \theta(y;\alpha,\lambda) = \lim_{\alpha'\uparrow\alpha} \theta(y;\alpha',\lambda)$.
\end{proof}

To deduce the continuity with respect to $\alpha$ and $\lambda$ in the directions not covered by the previous lemma, we need to make a 
case distinction. This depends on whether or not the certain points around which we want to show continuity are points where percolation 
occurs.
\begin{lemma}\label{lem:thetactsbelow}
Let $\alpha, \lambda >0$.
Suppose that there exists a $\lambda_0 < \lambda$ such that $\Pee_{\alpha,\lambda_0}(\text{percolation} ) > 0$.
Then for all $y\geq 0$ we have  $\lim_{\lambda' \uparrow \lambda} \theta( y; \alpha, \lambda' ) = \theta(y;\alpha, \lambda)$.

Similarly, if there exists an $\alpha_0 > \alpha$ such that $\Pee_{\alpha_0,\lambda}(\text{percolation})>0$, then 
$\lim_{\alpha' \downarrow \alpha} \theta( y; \alpha',\lambda ) = \theta(y;\alpha,\lambda)$, for all $y \geq 0$.
\end{lemma}

\begin{proof} 
The proof is an adaptation of a proof by Van den Berg and Keane~\cite{BergKeane} for standard bond percolation. 
(see also Lemma 8.10 in~\cite{GrimmettBoek}, p.204).
Throughout the proof we consider the coupling provided by Lemma~\ref{lem:couplingPoisson} that ensures that 
a.s.~$\Pcal_{\alpha',\lambda'} \subseteq \Pcal_{\alpha'',\lambda''}$ whenever $\alpha' \geq \alpha''$ and $\lambda' \leq \lambda''$.

We start by proving the first statement of the lemma.
Let $E_{\lambda'}$ denote the event that $(0,y)$ is in an infinite component of $\Gamma (\Pcal_{\alpha,\lambda'}\cup \{ (0,y) \})$.
Observe that $E_{\lambda} \supseteq \bigcup_{\lambda' < \lambda} E_{\lambda'}$.
Since $\lim_{\lambda' \uparrow \lambda} \Pee( E_{\lambda'} ) = \Pee( \bigcup_{\lambda' < \lambda} E_{\lambda'} )$, it suffices
to show that 

\[ 
\Pee( \bigcup_{\lambda' < \lambda} E_{\lambda'} ) = \Pee( E_\lambda ).
\]

Aiming for a contradiction, let us assume that 
$\Pee( E_\lambda \cap \bigcap_{\lambda' < \lambda} E_{\lambda'}^c ) > 0$, and
let us consider a realization of our marked Poisson process $\Pcal_{\alpha,\lambda}$ for which
$E_\lambda \cap \bigcap_{\lambda' < \lambda} E_{\lambda'}^c$ holds.

Note that, a.s., in $\Gamma_{\alpha,\lambda_0}$ there is an infinite component.
If $E_{\lambda}$ holds, then there is a finite path 
$p_0 = (0,y), p_1,\dots, p_K \in \Pcal_{\alpha,\lambda} \cup \{(0,y)\}$ that 
connects $(0,y)$ with a vertex $p_K$ in the infinite component of $\Gamma_{\alpha,\lambda_0}$.
(Note that by Lemma~\ref{lem:therecanbeonlyone} there is only one infinite component in $\Gamma_{\alpha,\lambda}$, so such a path exists a.s.)
If $\Qcal$ is the intensity one Poisson process on $\eR \times [0,\infty)^2$ used in the construction of the coupling
in Lemma~\ref{lem:couplingPoisson}, then 
there are points $(x_1,y_1,z_1) , \dots, (x_K, y_K, z_K) \in \Qcal$ such that 
$p_i = (x_i, y_i)$ and $z_i < \lambda e^{-\alpha y_i}$ for all $i=1,\dots, K$.
So in particular, there must be a $\lambda_0 \leq \lambda' < \lambda$ such that 
$z_i < \lambda' e^{-\alpha y_i}$ for all $i=1,\dots, K$. This implies
that $p_1, \dots, p_K \in \Pcal_{\alpha,\lambda'}$ and hence $E_{\lambda'}$ holds 
for some $\lambda' < \lambda$. Contradiction!
This proves that $\Pee( E_\lambda \cap \bigcap_{\lambda' < \lambda} E_{\lambda'}^c ) = 0$ after all, and hence the lemma.

The proof of the second part is completely analogous. 
\end{proof}

Finally, we need to consider the case where $\alpha=1, \lambda < \lambdac$ and $\alpha'$ approaches $1$ from above. 
To this end, we will need a lemma in which we approximate the event that $(0,y)$ does not lie in an infinite component by the event that 
the component of $(0,y)$ induced within a large but bounded region is small. 

More specifically, for $h \geq n \geq y \geq 0$ we define the event $U(y;n,h)$ as follows

\[ U(y;n,h) := \left\{ \begin{array}{l}
\text{In $\Gamma\left( (\Pcal_{\alpha,\lambda} \cup \{(0,y)\})\cap [-e^h, e^h] \times[0,h] \right)$, the component of $(0,y)$} \\ 
\text{   is contained in $[-n,n]\times[0,n]$ and has at most $n$ vertices}
\end{array} \right\}.
\]

\begin{figure}[h]
\begin{center}
\includegraphics[scale = 0.5]{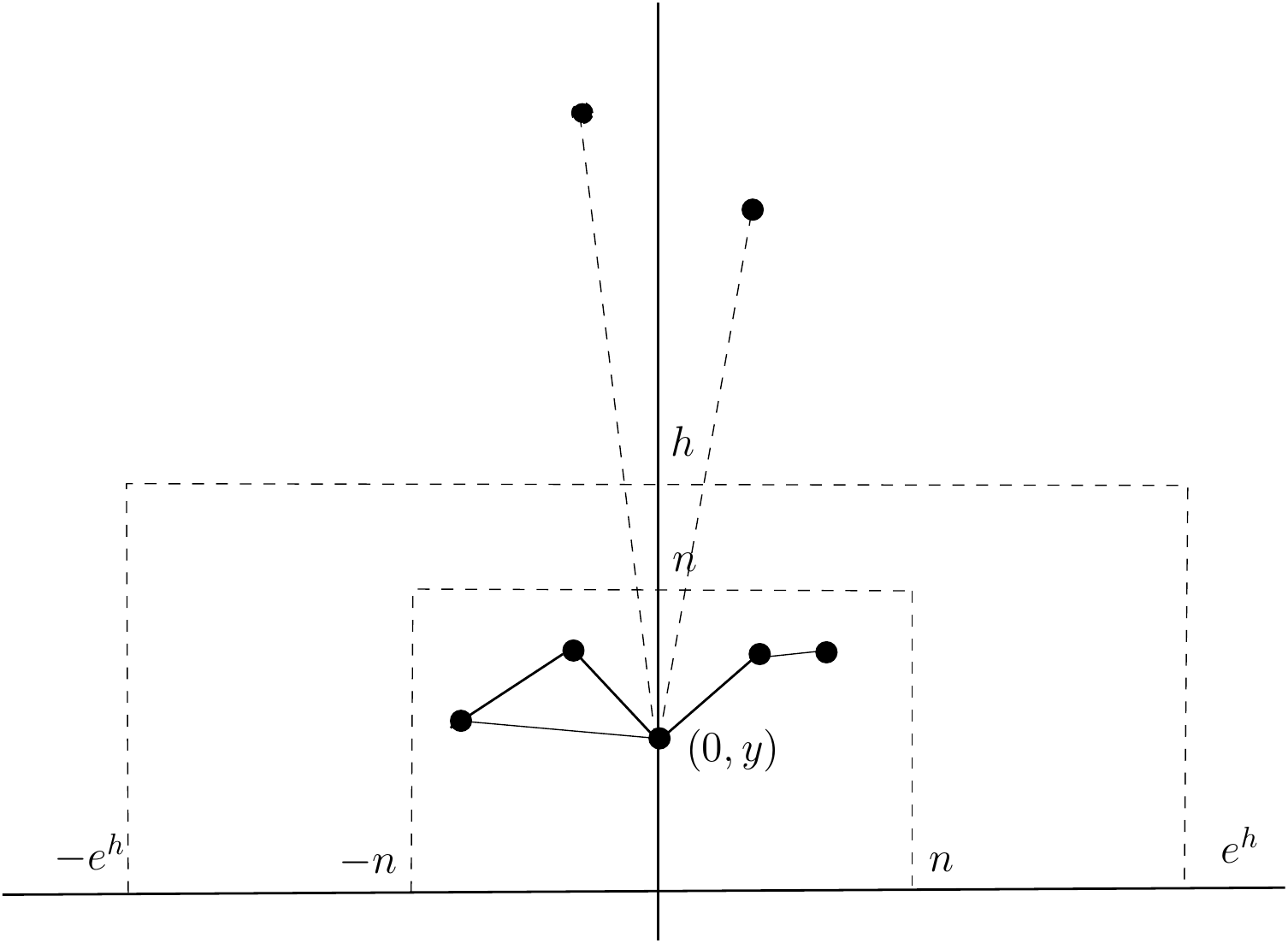}
\end{center}
\caption{The event $U(y;n,h)$.\label{fig:U}}
\end{figure}

\begin{lemma}\label{lem:U}
For every $\alpha > 1/2, \lambda > 0$ and $K, \eps > 0$, there exists an $n_0 = n_0(\alpha,\lambda,K,\eps)$ 
such that

\[ \Pee_{\alpha,\lambda}( U(y;n,h) ) \geq 1-\theta(y;\alpha,\lambda) - \eps, \]

\noindent
for all $h \geq n \geq n_0$ and all $0 \leq y \leq K$.
\end{lemma}

\begin{proof}
Let $E(y)$ denote the event that $(0,y)$ is in a finite component of $\Gamma(\Pcal_{\alpha,\lambda} \cup \{(0,y)\})$, and let 
$E(y,n)$ be the event that this component has at most $n$ vertices and is contained in $[-n,n]\times[0,n]$.
Clearly $E(y)=\bigcup_{n} E(y,n)$, so that there exists an $n_0 = n_0(y)$ such that%
\[ \Pee_{\alpha, \lambda} ( E(y,n) ) \geq \Pee_{\alpha,\lambda} ( E(y) ) - \eps/3 = 1 - \theta(y;\alpha,\lambda) - \eps/3, \]

\noindent
for all $n \geq n_0$.
Recall that $\theta$ is continuous in $y$ by Lemma~\ref{lem:thetactsy}.
By an almost verbatim repeat of the proof of Lemma~\ref{lem:thetactsy}, we have that $y \mapsto \Pee_{\alpha,\lambda} ( E(y,n) )$ is continuous in $y$ for all fixed $n$.
Thus, for every $y \in [0,K]$ there is a $\delta(y)$ such that 
$\Pee( E(y', n_0(y) ) \geq 1 - \theta(y';\alpha,\lambda) - 2\eps/3$ for all 
$y' \in (y-\delta(y),y+\delta(y'))$.
By compactness, there exist $y_1, \dots, y_M \in [0,K]$ such that 
$[0,K] \subseteq \bigcup_{i=1}^M (y_i-\delta(y_i),y_i+\delta(y_i))$.
Let us now set $n_0 := \max( n_0(y_1), \dots, n_0(y_M) )$.
Then we have that, for every $n \geq n_0$ and all $0\leq y \leq K$, $\Pee_{\alpha,\lambda} ( E(y,n) ) \geq 1 - \theta(y; \alpha,\lambda) - 2\eps/3$.
(Since each $y \in [0,K]$ is in some interval $(y_i-\delta(y_i),y_i+\delta(y_i))$ and hence
$\Pee_{\alpha,\lambda}  ( E(y,n) ) \geq \Pee_{\alpha,\lambda}  ( E(y,n_0(y_i) ) \geq 1 - \theta(y; \alpha,\lambda) - 2\eps/3$.)

To conclude the proof, we simply remark that $U(y,n,h) \supseteq E(y,n)$ for all every $h\geq n$. 
\end{proof}

\begin{lemma}\label{lem:thetactsifnoperc}
If $\Pee_{1,\lambda}(\text{percolation}) = 0$, then
$\lim_{\alpha'\to 1} \theta(y;\alpha',\lambda) = \theta(y;1,\lambda) = 0$ for all $y\geq 0$. 
\end{lemma}

We remark that, since we do not know whether or not there is percolation a.s.~when $\lambda=\lambdac$, we
do not just want to change the prerequisite $\Pee_{1,\lambda}(\text{percolation}) = 0$ into 
$\lambda<\lambdac$. If in a future work, it turns out that there is no percolation a.s.~at $\lambda=\lambdac$ then the 
lemma applies.

\begin{proof}
Let us fix some $y \geq 0$ and let $\eps > 0$ be arbitrary.
Using the previous lemma, we select an $n$ such that $\Pee_{1,\lambda}( U(y,n,h) ) \geq 1-\eps/3$
for all $h\geq n$.
Let

\begin{equation*}
\begin{array}{rcl}
R_h &=&  
\{ (x,y) \in \eR\times[0,\infty) \setminus [-e^h,e^h]\times[0,h] : \exists (x',y') \in [-n,n]\times[0,n]
\text{ such that } |x-x'| \leq e^{\frac12(y+y')} \} \\
&=& \{ (x,y) \eR\times[0,\infty) \setminus [-e^h,e^h]\times[0,h] : |x| \leq n+e^{\frac12(y+n)} \}.
\end{array}
\end{equation*}

Note that, if we keep $n$ fixed and make $h$ sufficiently large, we have $n+e^{\frac12(y+n)} < e^h$.
Thus, for sufficiently large $h$, we have $R_h \subseteq \{y \geq h\}$.
Denoting by $A_h$ the event that $R_h \cap \Pcal_{\alpha, \lambda} \not = \emptyset$, we see that 

\[ 
 \Pee_{\alpha,\lambda}( A_h ) \leq \Ee |R_t \cap \Pcal_{\alpha, \lambda} | 
 \leq \int_h^{\infty} 2(n+e^{\frac12(n+t)})\lambda e^{-\alpha t}dt  
 = \frac{2\lambda n}{\alpha}e^{-\alpha h} + \frac{2\lambda e^{n/2}}{\alpha-\frac12} e^{(\frac12-\alpha)h}.
\]

\noindent
Thus, we can fix a $h = h(\eps,n, \lambda)$ sufficiently large for 
$\Pee_{\alpha,\lambda}( A_h ) < \eps/2$ to hold uniformly for all $\alpha > .9$.

By Lemma~\ref{lem:VerwBegrCont}, there exists a $\delta>0$ such that 
$|\Pee_{\alpha,\lambda}( U(y;n,h) ) - \Pee_{1,\lambda}( U(y;n,h) )| \leq \eps/2$ for all $\alpha \in (1-\delta,1+\delta)$.
Let $E$ denote the event that $(0,y)$ is in an infinite component.
Then $E \subseteq U(y;n,h)^c \cup A_h$.
It follows that, for all $\alpha \in (1-\delta,1+\delta)$, 

\[ \Pee_{\alpha,\lambda}( E ) < \eps. \]

\noindent
Since $\eps>0$ was arbitrary, the result follows.
\end{proof}




\subsubsection*{The continuity of $c(\alpha, \nu)$.}

Here we briefly spell out how part~\ref{itm:main.v} of Theorem~\ref{thm:main} follows from the lemmas we have proved in this section.
That $c$ is continuous in the points claimed in property~\ref{itm:main.v} follows from Lemmas~\ref{lem:thetactsabove},~\ref{lem:thetactsbelow},
~\ref{lem:thetactsifnoperc} together with properties~\ref{itm:main.i}--\ref{itm:main.iv}, using the dominated convergence theorem.
(To apply the dominated convergence theorem, we note that $\theta(y,\alpha,\nu\alpha/\pi) \alpha e^{-\alpha y}$ is continous in $y$ by 
Lemma~\ref{lem:thetactsy}, and hence measurable, and that it is majorized by the integrable function $e^{-\alpha_0 y}$ with $0 < \alpha_0 < \alpha$.)
That $c$ has a point of discontinuity on every $(\alpha,\nu) \in \{1\}\times (\nuc,\infty)$ follows from the fact that 
$c(1,\nu) > 0$ for $\nu>\nuc$ but $c(\alpha,\nu) = 0$ for all $\alpha > 1$.

\section{Transferring to $G(N;\alpha, \nu)$: proof of Theorem~\ref{thm:main}}\label{sec:transfer}

\subsection{Finitary approximation of the percolation probability}
In this section, we prove some preliminary lemmas that will enable us to transfer the behaviour of the continuum percolation model to 
the finite random graph $G(N;\alpha, \nu)$. Much as in Lemma~\ref{lem:U} above, these lemmas approximate the probability 
$\theta (y;\alpha, \lambda)$ that the point $(0,y)$ belongs to an infinite component by events that are determined within a large but bounded domain. 

For $y, w, h \geq 0$, we define the event $T(y;h, w)$ as follows

\[ T(y;h,w) := \begin{array}{l}{\big\{}
\text{In $\Gamma(\Pcal_{\alpha,\lambda} \cup \{(0,y)\})$, there is a path between $(0,y)$ and a point in $\eR \times [h, 2h]$,} \\ 
\text{  and all points of this path lie in $[-w e^h, w e^h]\times[0,2h]$}{\big\}}
\end{array} 
\]

\noindent
See Figure~\ref{fig:Phw} for an illustration of the event $T(y;h,w)$.

\begin{figure}[h]
\begin{center}
\input{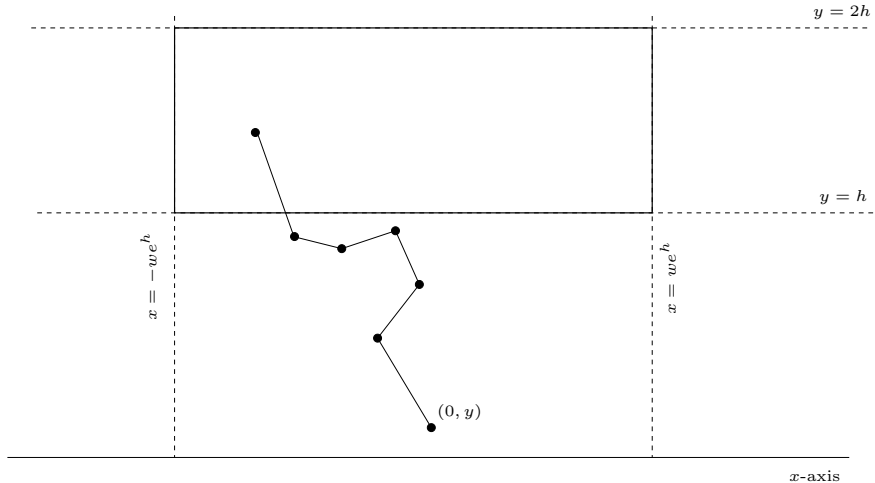}
\end{center}
\caption{Illustration of the event $T(y;h,w)$.\label{fig:Phw}}
\end{figure}

\begin{lemma}\label{lem:T}
Fix $\frac12 < \alpha \leq 1$ and $\lambda > 0$ such that $\Pee_{\alpha,\lambda}(\text{percolation}) = 1$. 
For every $K,\eps>0$ there exist constants $w = w(\alpha,\lambda,K,\eps), h_0 = h_0(\alpha,\lambda,K,\eps) > 0$ such that

\[ 
\Pee_{\alpha,\lambda}( T(y;w,h) )  
\geq \theta(y;\alpha,\lambda) - \eps, \]

\noindent
for all $0 \leq y \leq K$ and all $h \geq h_0$.
\end{lemma}

\begin{proof} 
For notational convenience we will denote $T(y;w,h)$ simply by $T$ in the proof.
Let $E$ denote the event that there is an infinite path starting at $(0,y)$. 

Let us denote by $\Ccal$ the component of $(0,y)$ in the 
graph $\Gamma( (\Pcal_{\alpha,\lambda} \cup \{(0,y)\}) \cap [-we^h,we^h]\times[0,h])$, i.e.~the subgraph 
induced by the points inside $[-we^h,we^h]\times[0,h]$.
We set $h' := h(2\alpha-1) + \ln(w/2)$, 
and 

\[ \begin{array}{rcl}
    A_1 & := & \left\{ \Ccal \subseteq [-we^h/2,we^h/2]\times[0,h']\right\}, \\
    A_2 & := & \left\{ \exists p \in \Ccal \text{ with } |x(p)| > we^h/2 \right\}, \\
    A_3 & := & A_1^c \cap A_2^c.
   \end{array} \]

\noindent
Clearly $\Pee(A_1)+\Pee(A_2)+\Pee(A_3)=1$. As usual by new we let $E$ denote the event that $(0,y)$ belongs to an infinite component. 

In the discussion which follows, it is helpful to think of the situation where we uncover $\Pcal_{\alpha,\lambda}$ in two stages.
First we reveal only the points inside $[-we^h,we^h]\times[0,h]$ and then we reveal the rest of the points.

We first observe that if $A_1$, $E$ and $T^c$ all hold, then there must be a point of $\Pcal_{\alpha,\lambda}$ in the set $U$
of all points in $\eR \times [0,\infty) \setminus [-we^h,we^h]\times[0,2h]$ that could possibly be connected to a point in 
$[-we^h/2,we^h/2]\times[0,h']$.
We have

\[ \begin{array}{rcl} U & = & 
[-we^h,we^h]\times(2h,\infty) \cup \{ (x,y) \in \eR \times [0,\infty) : we^h < |x| < e^{\frac12(y+h')} + we^h/2 \} \\
& \subseteq & 
\{ (x,y) \in \eR \times [0,\infty) : |x| < 2e^{\frac12(y+h')}, y > (3-2\alpha)h+\ln(w/2) \},
  \end{array} \]

\noindent
using that $y=(3-2\alpha)h+\ln(w/2)$ solves the equation $e^{\frac12(y+h')} = we^h/2$, and that $(3-2\alpha)h+\ln(w/2) < 2h$ 
if $h$ is sufficiently large with respect to $w$.
Hence, for such $h$ and $w$, we have

\[ \begin{array}{rcl}
\Ee | U \cap \Pcal_{\alpha, \lambda} |  
& \leq & 
\displaystyle
2 \lambda \int_{(3-2\alpha)h+\ln(w/2)}^\infty
e^{\frac12(y+h')} e^{-\alpha y}{\dd}y \\
& = & 
\displaystyle
2\lambda e^{\frac12 h' + (\frac12-\alpha)( (3-2\alpha)h + \ln(w/2))} / \left( \alpha-\frac12 \right) \\
& = & 
2 \lambda e^{(2\alpha-1)(\alpha-1)h + (1-\alpha)\ln(w/2)} / \left( \alpha-\frac12 \right).
   \end{array}
\]

\noindent
Note that the coefficient of $h$ in the last expression is negative.
It follows that, for every $w, \eps$ we can find a $h_0 = h_0(w, \eps)$ such that 

\begin{equation}\label{eq:ETc1} 
\Pee( E\cap T^c | A_1 ) \leq \Ee | U \cap \Pcal_{\alpha, \lambda} | < \eps,  
\end{equation}

\noindent
for all $h \geq h_0(w,\eps)$.

Next, we observe that if $A_2$ holds and there is at least one point in 
$[0,we^h/2]\times [h,2h]$ and at least one point in $[-we^h/2,0]\times [h,2h]$ then 
$T$ holds by part~\ref{itm:cross.i} of Lemma~\ref{lem:cross}.
This implies that

\begin{equation}\label{eq:ETc2} 
\Pee( E \cap T^c | A_2 ) \leq \Pee( T^c | A_2 ) \leq 2 \exp[ - \lambda we^h(e^{-\alpha h}-e^{-2\alpha h})/2\alpha ]
= \exp[ - \Omega( w ) ], 
\end{equation}

\noindent
using that $\alpha \leq 1$ in the last equality.
Hence, we can choose $w$ such that this conditional probability is strictly less than $\eps$, no matter what the value of $h\geq 0$ is.

Suppose that $A_3$ holds, and let $p \in \Ccal$ be a point with $y(p) \geq h'$.
Any point in $U := [x(p)-e^{\frac12(h'+h)}, x(p)+e^{\frac12(h'+h)}] \times [h,h+1]$ will be adjacent to $y(p)$.
Note that this set is contained in $[-we^h,we^h]\times[h,2h]$, since we must have $|x(p)| \leq we^h/2$ 
and $e^{\frac12(h'+h)} = \sqrt{w/2} e^h < we^h/2$.
It thus follows that

\begin{equation}\label{eq:ETc3} 
\begin{array}{rcl} 
\Pee( E \cap T^c | A_3 ) 
& \leq & 
\Pee( T^c | A_3 )  \\
& \leq & 
\exp[ - 2\lambda e^{\frac12(h'+h)} e^{-\alpha h}(1 - e^{-\alpha})/\alpha ] \\
& = &
\exp[ - \Omega( e^{\frac12\ln(w/2)} ) ]  \\
&  = & 
\exp[ - \Omega( \sqrt{w} ) ],
\end{array} \end{equation}

\noindent
Again, we can choose $w$ such that this conditional probability is at most $\eps$, no matter what the value of $h\geq 0$ is.

Combining~\eqref{eq:ETc1},~\eqref{eq:ETc2} and~\eqref{eq:ETc3}, we find that

\[ \Pee( E \cap T^c )
 = \Pee( E \cap T^c | A_1 )\Pee(A_1) + \Pee( E \cap T^c | A_2 )\Pee(A_2) + \Pee( E \cap T^c | A_3 )\Pee(A_3) \leq \eps,
\]

\noindent
for $w$ sufficiently large and all $h \geq h_0(w,\eps)$.
This in turn gives that, for such pairs $w,h$: 

\[ \begin{array}{rcl} 
\Pee( T ) & \geq & 
\Pee( E \cap T ) \\
& = &
\Pee( E ) - \Pee( E \cap T^c ) \\
& \geq & 
\Pee(E) - \eps \\
& = & 
\theta(y;\alpha,\lambda) - \eps.
\end{array} \]

\noindent
To conclude the proof, we simply note that each of the bounds~\eqref{eq:ETc1},~\eqref{eq:ETc2} and~\eqref{eq:ETc3} holds uniformly over all
$y \leq h_0$.
\end{proof}

Let $C_{w,h}$ denote the event that there exists a path in $\Gamma(\Pcal_{\alpha,\lambda})$ starting
at a point of $\Pcal_{\alpha,\lambda} \cap [-we^h,-(w-1)e^h] \times [0,h]$ and ending at a point of
$\Pcal_{\alpha,\lambda} \cap \{ [(w-1)e^h, we^h] \times[0,h]\}$, with all its points having $y$-coordinate at most $h$.
See Figure~\ref{fig:pad} for a depiction.
\begin{figure}[!h]
\begin{center}
\input{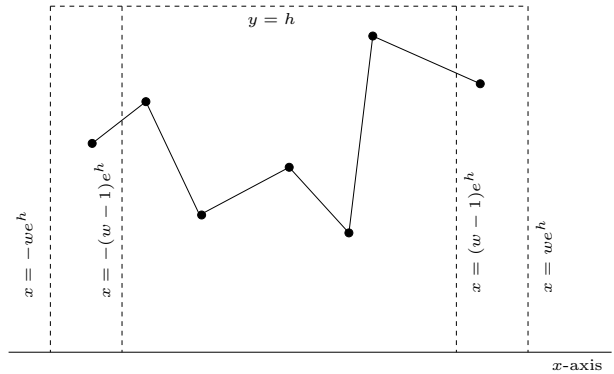}
\end{center}
\caption{Depiction of the event $C_{w,h}$ (not to scale).\label{fig:pad}}
\end{figure}
We will show that if the parameters $\alpha$ and $\lambda$ are such that percolation occurs with probability 1, then 
as $h$ grows the probability of $C_{w,h}$ converges to 1. 
To this end, we will need the following lemma which states that the infinite component extends in all directions indefinitely. 

\begin{lemma}\label{lem:percogoesleftright}
 Suppose that $\alpha,\lambda$ is such that $\Pee_{\alpha,\lambda}(\text{percolation} ) = 1$.
 Then, for every vertical line $\ell$, the unique infinite component
 contains points on either side of $\ell$.
\end{lemma}

\begin{proof}
We define 

\[ E_n := \{\text{the unique infinite component is contained in the halfspace $\{x\geq n\}$}\}. \]

\noindent
 Let us write $c := \Pee( E_0 )$.
 Since the model is invariant under horizontal translations, we have $\Pee( E_n ) = \Pee( E_0 )$ for all $n \in \Zed$.
 Hence
 
 \[ c = \lim_{n\to\infty} \Pee( E_n ) = \Pee\left( \bigcap_{n\in\eN} E_n\right) = 0, \]
 
 \noindent
 as clearly $\bigcap_n E_n = \emptyset$.
 So $\Pee( E_n ) = 0$ for all $n$.  \\
 In other words, almost surely, for every vertical line $\ell$, the infinite component contains points to the left of $\ell$.
 By symmetry, it also contains points to the right of every $\ell$ almost surely.
\end{proof}

%
%

We can now proceed with the proof of the statement regarding $C_{w,h}$.

\begin{lemma}\label{lem:Ch}
 Let $\alpha,\lambda > 0$ be such that $\Pee_{\alpha,\lambda}(\text{percolation}) =1$.
 For every (fixed) $w>1$, we have $\lim_{h\to\infty} \Pee_{\alpha,\lambda}( C_{w,h} ) = 1$.
\end{lemma}

\begin{proof}
%
We begin with the case where $\alpha < 1$. Here, we will show something stronger: with probability approaching 1 as $h\to \infty$
there is a path between two points of $\Pcal_{\alpha,\lambda}$ whose $x$-coordinates belong to $[-we^h,-(w-1)e^h]$ and $[(w-1)e^h,we^h]$, 
respectively, and the remaining points have $y$-coordinates that are between $h-1$ and $h$.  
To this end, we define a collection of boxes $B_i=[i e^{h-1}/2, (i+1)e^{h-1}/2) \times [h-1,h]$, for $i=-\lfloor 2we \rfloor, \ldots, 
\lfloor 2we\rfloor$. Note that the leftmost and the rightmost boxes are such that the points of both have $x$-coordinates which 
belong to $[-we^h,-(w-1)e^h]$ and $[(w-1)e^h,we^h]$, respectively.

We will show that 1) any points of $\Pcal_{\alpha, \lambda}$ that belong to adjacent boxes must be also adjacent
and 2) with high probability all boxes contain at least one point.   

To show 1), consider a point $p_1 \in B_0$ and a point $p_2 \in B_1$. Then $|x(p_1) - x(p_2)|\leq 2 e^{h-1}/2 = e^{h-1}\leq 
e^{\frac{y(p_1)+y(p_2)}{2}}$, since $y(p_1),y(p_2)\geq h-1$. 

Part 2) follows from a simple calculation. We have 
\begin{equation*}
\begin{split}
\Ee |B_0 \cap \Pcal_{\alpha, \lambda}| &= \lambda \int_{-(h-1)}^{h} \int_0^{e^{h-1}/2} e^{-\alpha y} dx dy 
= \frac{\lambda}{2e} e^h \int_{-(h-1)}^{h} e^{-\alpha y} dy\\ 
& =\frac{\lambda}{2e\alpha} e^h (e^{-\alpha(h-1)} - e^{\alpha h})= \frac{\lambda}{2e\alpha} e^{(1-\alpha)h} (e^{\alpha} - 1),
\end{split}
\end{equation*}
whereby 
$$ \Pee_{\alpha, \lambda} (|B_0 \cap \Pcal_{\alpha, \lambda}| = 0) = 
\exp \left( - \frac{\lambda (e^{\alpha} - 1)}{2e\alpha} e^{(1-\alpha)h}  \right).$$
Hence, provided $\alpha < 1$, the probability that at least one of these boxes does not contain a point is at most 
$(4we +1) \exp \left( - \frac{\lambda (e^{\alpha} - 1)}{2e\alpha} e^{(1-\alpha)h}  \right) \to 0, \ \mbox{as $h\to \infty$}$. 

We now focus on the case where $\alpha =1$. 
Let $E_h$ denote the event that $[-h,h]\times[0,h]$ contains a point of the infinite component of $\Gamma_{\alpha,\lambda}$, and this component
has a point $p_1$ with $x(p_1) < -w e^h$ and a point $p_2$ with $x(p_2) > w e^h$.
By Lemma~\ref{lem:percogoesleftright}, we have $\bigcup_h E_h = \{\text{percolation}\}$.
So, in particular,
\begin{equation}\label{eq:Ehlim} 
\lim_{h\to\infty} \Pee_{\alpha, \lambda} ( E_h ) = 1.  
\end{equation}

\noindent
Define $A_h := \{ (x,y) : y \geq h, |x| \leq w e^h + e^{\frac12(h+y)}\}$.
That is, $A_h$ is the set of all points with $y$-coordinate at least $h$ that could be adjacent to 
some point in $[-w e^h,w e^h]\times[0,h]$.
Similarly to what we did in the proof of Lemma~\ref{lem:thetactsabove}, we compute:

\[ \begin{array}{rcl} 
\Ee |\Pcal_{1,\lambda}\cap A_h|
& = & \lambda \int_{A_h} e^{-\alpha y}{\dd}x{\dd}y \\
& = & \lambda \int_h^\infty\int_{-(we^h+e^{\frac12(h+y)}}^{(we^h+e^{\frac12(h+y)}} e^{- y}{\dd}x{\dd}y \\
& = & \lambda \left( 2w e^h \int_h^\infty e^{-y}{\dd}y + 2 e^{h/2} \int_h^\infty e^{-y/2}{\dd}y \right) \\
& = & (2w+1) \lambda.
\end{array} \]

\noindent
Let $F_h$ denote the event that the area $A_h$ does not contain any point of $\Pcal_{\alpha, \lambda}$.
Then we have

\[ \Pee_{1,\lambda}( F_h ) = e^{-(2w+1) \lambda}. \]

\noindent
Let us also remark that $E_h \cap F_h \subseteq C_{w,h}$. 
(If $E_h\cap F_h$ holds, then there is a path with all $y$-coordinates at most $h$, between
a vertex $p_1$ with $ x ( p_1)< we^h$ and a vertex $p_2$ with $x(p_2) >we^h$. 
The $x$-coordinates of any two adjacent vertices of this path differ by no more than 
$e^h$, so there must be points in $[-we^h,-(w-1)e^h]\times[0,h]$ and $[(w-1)e^h,we^h]\times[0,h]$.)
Observe that $C_{w,h}$ and $F_h$ are independent, since they depend on the points in disjoint parts of the plane.
Thus,

\[ \begin{array}{rcl}
\Pee(C_{w,h}) e^{-(2w+1) \lambda} & = &  
\Pee(C_{w,h}) \Pee(F_h) \\
& = & \Pee(C_{w,h} \cap F_h) \\ 
& \geq & \Pee( E_h \cap F_h ) \\
& \geq & \Pee( E_h ) - (1-\Pee(F_h)) \\
& = & \Pee( E_h ) + e^{-(2w+1) \lambda} - 1. 
\end{array} \]

\noindent
That $\lim_{h\to\infty} \Pee( C_{w,h} ) = 1$ 
now follows immediately from~\eqref{eq:Ehlim}.
\end{proof}

\subsection{Approximating the KPKVB model and the proof of Theorem~\ref{thm:main}}
We are now ready to prove Theorem~\ref{thm:main}, by establishing the link between the continuum percolation model
in the previous section and the KPKVB model.
It only remains to show that, for all $\alpha, \nu > 0$, we have that
$|\Cscr_{(1)}|/N \to c(\alpha,\nu), |\Cscr_{(2)}| /N\to 0$ in probability, where $\Cscr_{(1)}$ and $\Cscr_{(2)}$ denote the largest and the 
second largest component of $G(N;\alpha,\nu)$. 
Note that for $\alpha > 1$, we have already proved this in our ealier paper~\cite{BFMgiantEJC}.
Let us also remark that, since $c$ is continuous for $\alpha < 1$ and $c=1$ for $\alpha \leq 1/2$, by the monotonicity in $\alpha$ of 
$G=G(N;\alpha,\nu)$ (see Lemma~1.2 from~\cite{BFMgiantEJC}) it suffices to consider only the 
case $\alpha > \frac12$.
In the remainder of this section we shall thus always assume that $1/2 <\alpha \leq 1$.

\smallskip

Let $\GPo = \GPo(N;\alpha,\nu)$ denote the random graph which is defined just as the original KPKVB-model $G = G(N;\alpha,\nu)$ 
with the only difference that now we drop $Z \isd \Po(N)$ points onto the hyperbolic plane according
to the $(\alpha, R)$-quasi uniform distribution, where $Z$ is of course independent of the locations of these points.
Note that this also gives a natural coupling between $G$ and $\GPo$ : if $X_1, X_2, \dots$ is an infinite 
supply of points taken i.i.d.~according to the $(\alpha, R)$-quasi uniform distribution, then $G$
has vertex set $\{X_1,\dots, X_N\}$ while $\GPo$ has vertex-set $\{X_1,\dots, X_Z\}$. 
The following lemma shows it is enough to show Theorem~\ref{thm:main} with $\GPo$ in place of $G$.

\begin{lemma}\label{lem:PoissonApproxIsOk}
Suppose there is a constant $t$ such that $|\Cscr_{(1)}(\GPo)| = (t+o_{p}(1))N$ and $|\Cscr_{(2)}(\GPo)| = o_{p}(N)$.
Then also $|\Cscr_{(1)}(G)| = (t+o_p(1))N$ and $|\Cscr_{(2)}(G)| = o_{p}(N)$.
\end{lemma}

\begin{proof}
Aiming for a contradiction, suppose that $\limsup_{N \to \infty} \Pee( |\Cscr_{(1)}(G)| > (1+\eps)t ) > 0$ for some $\eps>0$.
Recall that $\Pee( Z \geq N ) = 1/2 + o(1)$ (by the central limit theorem, for example) and observe that
whenever $Z \geq N$ we have that $\GPo \supseteq G$ (under the natural coupling specified just before the statement of this lemma).
But then we also have%
\[ \limsup_{N \to \infty} \Pee \left( \frac{|\Cscr_{(1)}(\GPo)|}{N} > (1+\eps)t \right) \geq 
\limsup_{N \to \infty} (1/2+o(1)) \cdot  \Pee\left( \frac{|\Cscr_{(1)}(G)|}{N} > (1+\eps)t \right) > 0, \]

\noindent
a contradiction!

Completely analogously, we cannot have that $\limsup_{N\to \infty} 
\Pee( |\Cscr_{(1)}(G)|\cdot N^{-1} < (1-\eps)t ) > 0$ for some $\eps>0$.
Applying the same argument to compare $|\Cscr_{(1)}(\GPo) \cup \Cscr_{(2)}(\GPo)|$ to 
$|\Cscr_{(1)}(G) \cup \Cscr_{(2)}(G)|$ completes the proof.
\end{proof}

In the remainder of this section, we will thus restrict attention to proving Theorem~\ref{thm:main} with $\GPo$ in place of
$G$ (under the additional assumption that $1 \geq \alpha > \frac12$).

Next we define a correspondence between the continuum percolation model $\Gamma_{\alpha,\lambda}$ and 
$\GPo$.
Let us define $\Psi : [0,R]\times (-\pi,\pi] \to (-\frac{\pi}{2} e^{R/2}, \frac{\pi}{2} e^{R/2}] \times [0,R]$ by:

\[ \Psi : (r,\vartheta) \mapsto (\vartheta \cdot \frac{e^{R/2}}{2} , R-r). \]

\noindent
We let $V$ denote the vertex set of $\GPo$, and we let 
$\Vtil$ denote $\Pcal_{\alpha,\nu\alpha/\pi} \cap [-\frac{\pi}{2} e^{R/2}, \frac{\pi}{2} e^{R/2}]\times[0,R]$.
Let us denote by $\Gammatil$ the graph with vertex set $\Vtil$ and an 
edge between $(x,y), (x',y') \in\Vtil$ if and only if $|x-x'|_{\pi e^{R/2}} \leq e^{\frac12(y+y')}$.
(In other words, $\Gammatil$ is the supergraph of $\Gamma_{\alpha,\lambda}$ induced on $\Vtil$, together with some extra edges 
for ``wrap around''.)

\begin{lemma}\label{lem:coupling}
There exists a coupling such that, a.a.s., 
$\Vtil = \Psi(V)$.
\end{lemma}

\begin{proof}
$V$ constitutes a Poisson process on $(-\pi,\pi]\times[0,R]$ with intensity function:

\[ f_V(r,\vartheta) := N \cdot \frac{\alpha\sinh(\alpha r)}{\cosh(\alpha R)-1} \cdot \frac{1}{2\pi}
= 
\frac{\nu\alpha}{2\pi} \cdot e^{R/2} \cdot \frac{\alpha\sinh(\alpha r)}{\cosh(\alpha R)-1}. \]

\noindent
By the mapping theorem (see~\cite{KingmanBoek}, page 18), $\Psi(V)$ is a Poisson process on
$[-\frac{\pi}{2} e^{R/2}, \frac{\pi}{2} e^{R/2}]\times[0,R]$ with intensity function

\[ f_{\Psi(V)}(x,y) := f_V\left(\Psi^{-1}(x,y)\right) |\det(J)|, \]

\noindent
where $J$ denotes the Jacobian of $\Psi^{-1}$. It is easily checked that $|\det(J)| = 2 e^{-R/2}$.
We see that 

\[ \begin{array}{rcl}
f_{\Psi(V)}(x,y) 
& = & 
\frac{\nu\alpha}{\pi} \frac{\frac12(e^{\alpha(R-y)} - e^{\alpha(y-R)})}{\frac12(e^{\alpha R} - e^{-\alpha R}) - 1} \\
 & = & 
 \frac{\nu\alpha}{\pi} e^{-\alpha y} \cdot \left(
 \frac{1 - e^{2\alpha(y-R)}}{1-e^{-2\alpha R} - 2e^{-\alpha R}}
 \right) \\
 & = & 
 \frac{\nu\alpha}{\pi} e^{-\alpha y} \cdot \left(
 1 - O( e^{2\alpha(y-R)} ) + O( e^{-\alpha R} ) 
 \right).
\end{array} \]

\noindent
Let us also recall that $\Vtil$ is a Poisson process with intensity $f_{\Vtil}(x,y) = \frac{\nu\alpha}{\pi}e^{-\alpha y}$ on
$[-\frac{\pi}{2} e^{R/2}, \frac{\pi}{2} e^{R/2}]\times[0,R]$.
Let us write $f_{\min} := \min \{ f_{\Psi(V)}, f_{\Vtil} \}$.
Let $\Pcal_0, \Pcal_1, \Pcal_2$ be independent Poisson processes, $\Pcal_0$ with intensity $f_{\min}$, $\Pcal_1$ with intensity
$f_{\Psi(V)}-f_{\min}$ and $\Pcal_2$ with intensity $f_{\Vtil} - f_{\min}$.
We couple $\Vtil, \Psi(V)$ by setting $\Psi(V) = \Pcal_0 \cup \Pcal_1, \Vtil = \Pcal_0 \cup \Pcal_2$.
(This clearly also defines a coupling between $V$ and $\Vtil$.)
This way, the event $\Vtil = \Psi(V)$ coincides with the event $\Pcal_1=\Pcal_2 = \emptyset$.
Comparing $f_{\Psi(V)}$ and $f_{\Vtil}$, we see that 

\[ \begin{array}{rcl}
 \Ee |\Pcal_1|, \Ee |\Pcal_2| 
 & \leq & 
 \pi e^{R/2} \cdot O\left( \int_0^R e^{\alpha y - 2\alpha R}{\dd}y + \int_0^R
 e^{-\alpha (y + R)} {\dd} y \right) \\
 & = &
 O\left( e^{R(1/2-\alpha)} \right) \\
 & = & 
 o(1),
 \end{array} \]

\noindent
where we used the assumption that $\alpha > 1/2$.
This shows that, under the chosen coupling, $\Pee( \Vtil = \Psi(V) ) = 1-o(1)$, as claimed.
\end{proof}

Before we can continue studying $\GPo$ and $\Gammatil$, we first derive some useful asymptotics.

\begin{lemma}\label{lem:asymptotics}
There exists a constant $K>0$ such that, for every $\eps>0$ and for $R$ sufficiently large, the following holds.
Let us write $\Delta(r,r') := \frac12 e^{R/2} \arccos\left( (\cosh r \cosh r' - \cosh R)/ \sinh r \sinh r' \right)$.
For every $r,r'\in [\eps R,R]$ with $r+r' \geq R$ we have that 

\[ \frac12 e^{\frac12(y+y')} - K e^{ -\frac12\max\{ y,y'\} + \frac32\min\{y,y'\} - R}
 \leq \Delta( r, r') \leq \frac12 e^{\frac12(y+y')} + K e^{\frac32(y+y') - R},
\]
\noindent
where $y := R-r, y' := R-r'$.
\end{lemma}

\begin{proof}
We compute:

\[ \begin{array}{rcl}
\displaystyle 
 \frac{\cosh r \cosh r' - \cosh R}{\sinh r \sinh r'}
 & = & 
\displaystyle  
\frac{
 \frac14(e^{r+r'}+e^{r-r'}+e^{r'-r}+e^{-(r+r')}) -\frac12(e^R-e^{-R})
 }{
 \frac14(e^{r+r'}-e^{r-r'}-e^{r'-r}+e^{-(r+r')}) 
 } \\
 & = &
 \displaystyle 
 1
 + 
 \frac{
 2(e^{r-r'}+e^{r'-r})
 }{
 e^{r+r'}-e^{r-r'}-e^{r'-r}+e^{-(r+r')} 
 } \\
 & & 
 \displaystyle
 - \frac{2(e^R-e^{-R})}{ e^{r+r'}-e^{r-r'}-e^{r'-r}+e^{-(r+r')}
 } \\
 & = & 
 \displaystyle 
 1 + O\left(e^{-2\min\{ r,r'\} }\right) - 2 e^{R-(r+r')}\left(1+O\left(e^{-2\min\{ r,r'\}}\right)\right) \\
 & = & 
 \displaystyle 
 1 - 2 e^{R-(r+r')} \left(1 + O\left( e^{-2\min \{ r,r'\}} + e^{-2\min\{ r,r'\} + (r+r')-R}\right) \right) \\
 & = & 
 \displaystyle 
 1 - 2 e^{y+y'-R} \left(1 - O\left( e^{|y-y'|- R} \right) \right),
\end{array} \]

\noindent
using that $r + r' \geq R$ for the last line.
Since $\arccos(1-x) \geq \sqrt{2x}$ for all $0\leq x\leq 1$ (for completeness we provide a detailed derivation as Lemma~\ref{lem:arccoslemma} 
in the appendix), we see that

\[ \begin{array}{rcl}
\Delta (r,r')
& \geq &  
\frac12 e^{R/2} \cdot 2 e^{\frac12(y+y')-R/2} \cdot \sqrt{1-O\left( e^{|y-y'|- R} \right)} \\
& = & 
e^{\frac12(y+y')} \cdot \left(1-O\left( e^{|y-y'|- R} \right)\right)  \\
& = & 
 e^{\frac12(y+y')} - O\left(
 e^{ -\frac12\max\{ y,y'\} + \frac32\min\{ y,y'\} - R}
 \right),
\end{array} \]
 
 \noindent
where we have used that $\sqrt{1-x} = 1 - O(x)$ (see Lemma~\ref{lem:sqrtlemma} in the appendix).
This proves the lower bound in the lemma.

For the upper bound, we use that $\cos(1-x) \leq \sqrt{2x} + 1000 x^{3/2}$ (see Lemma~\ref{lem:arccoslemma} in the appendix). This shows that

 \[ \begin{array}{rcl}
 \Delta 
 & \leq & 
 e^{\frac12(y+y')} + \frac12 e^{R/2} \cdot 1000 e^{\frac32(y+y'-R)} \\
 & = & 
 e^{\frac12(y+y')} + 500 e^{\frac32(y+y') - R}.
\end{array} \]

\end{proof}

Let $X_1 = (r_1,\vartheta_1), X_2 = (r_2,\vartheta_2), \dots \in \Dcal_R$ be an infinite supply of i.i.d.~points
drawn according to the $(\alpha, R)$-quasi uniform distribution, and set 
$\Xtil_i := \Psi(X_i)$ for $i=1,2,\dots$.
For notational convenience, we will sometimes also write $\Xtil_i =: (x_i, y_i)$.
If the coupling from Lemma~\ref{lem:coupling} holds, we can write 
$V = \{ X_1, \dots, X_Z\}, \Vtil = \{ \Xtil_1, \dots, \Xtil_Z\}$.

We will make use of the following result, which is also known as the Mecke formula and can be found in Penrose's 
monograph~\cite{bk:Penrose}, as Theorem 1.6.

\begin{theorem}[\cite{bk:Penrose}]\label{thm:palm}
Let $\Pcal$ be a Poisson process on $\eR^d$ with intensity function $f$, and suppose that $\mu := \int f < \infty$.
Suppose that $h(\Ycal, \Xcal)$ is a bounded measurable function, defined on pairs $(\Ycal,\Xcal)$ with $\Ycal \subseteq \Xcal \subseteq \eR^d$
and $\Xcal$ finite, such that $h(\Ycal,\Xcal) = 0$ whenever $|\Ycal| \neq j$ (for some $j \in \eN$).
Then
\[ \Ee \sum_{\Ycal \subseteq \Pcal} h( \Ycal, \Pcal) = \frac{\mu^j}{j!} \cdot \Ee h(\{Y_1,\dots, Y_j\}, \{Y_1,\dots, Y_j\} \cup \Pcal ), \]
\noindent
where the $Y_i$ are i.i.d.~random variables that are independent of $\Pcal$ and have common probability density function
$f/\mu$.
\end{theorem}
We shall be applying the above theorem letting $f$ be the density function induced by $f_{\alpha, \nu \alpha /\pi}$ 
on $[-\frac{\pi}{2} e^{R/2},\frac{\pi}{2}e^{R/2}]\times [0,R]$. Thereby $\mu= \frac{\nu \alpha}{\pi} 
\int_{-\frac{\pi}{2} e^{R/2}}^{\frac{\pi}{2}e^{R/2}} \int_0^R e^{-\alpha y} dy dx = \nu e^{R/2}(1- e^{-\alpha R}) = N(1-o(1))$.

\begin{lemma}\label{lem:edges}
On the coupling space of Lemma~\ref{lem:coupling}, the following hold a.a.s.:
\begin{enumerate}
 \item\label{itm:edges.i} for all $i,j \leq Z$ we have $\Xtil_i\Xtil_j \in E(\Gammatil) \Rightarrow X_iX_j \in E(\GPo)$.
 \item\label{itm:edges.ii} for all $i,j \leq Z$ with $r_i + r_j \geq \frac32 R$, we have that
 $\Xtil_i\Xtil_j \in E(\Gammatil) \Leftrightarrow X_iX_j \in E(\GPo)$
\end{enumerate}
\end{lemma}

\begin{proof}
Note that if $r_i+r_j \leq R$ then $X_iX_i \in E(\GPo)$ by the triangle inequality.
So the lemma trivially holds for all pairs $i,j$ with $r_i+r_j \leq R$. 
Let us also remark that, a.a.s., there are no vertices $i$ with $r_i \leq \gamma R =: \frac12(1-\frac{1}{2\alpha}) R$.
(Since the expected number of such vertices is $O( e^{R/2} e^{-\alpha(1-\gamma) R} ) 
= O( e^{(1/2-\alpha + \alpha\gamma) R} ) = O( e^{\frac12 \left( \frac12 - \alpha \right)R} ) = o(1)$.)
In all the computations that follow we shall thus always assume that $r_i, r_j \geq \gamma R$ and $r_i+r_j > R$.

\smallskip

By the hyperbolic cosine rule, we have $X_iX_j \in E(\GPo)$ if and only if  $|\vartheta_i-\vartheta_j|_{2\pi} \leq
\Delta_{ij} := \arccos\left( (\cosh r_i \cosh r_j - \cosh R)/ \sinh r_i \sinh r_j \right)$. 
In other words, $X_iX_j \in E(\GPo)$ if and only if $|x_i-x_j|_{\pi e^{R/2}} \leq 
\Delta(r_i, r_j)$ with $\Delta(.,.)$ as in Lemma~\ref{lem:asymptotics}.

Let $A$ denote the number of pairs $\Xtil_i, \Xtil_j \in \Vtil$ for which $y_i+y_j \leq R$ and
$\Delta(r_i,r_j) < |x_i-x_j|_{\pi e^{R/2}} \leq e^{\frac12(y_i+y_j)}$.
By Theorem~\ref{thm:palm} and Lemma~\ref{lem:asymptotics}, we see that 

\[ \begin{array}{rcl}
\Ee A 
& = & 
\displaystyle
\int_{-\frac{\pi}{2}e^{R/2}}^{\frac{\pi}{2}e^{R/2}}\int_0^R\int_{-\frac{\pi}{2}e^{R/2}}^{\frac{\pi}{2}e^{R/2}}\int_0^{R-y_i} 
{\bf 1}_{\{\Delta(r_i,r_j) < |x_i-x_j|_{\pi e^{R/2}} \leq e^{\frac12(y_i+y_j)}\}}
\left(\frac{\nu\alpha}{\pi}\right)^2 e^{-\alpha y_i} e^{-\alpha y_j} {\dd}x_j{\dd}y_j{\dd}x_i{\dd}y_i \\
& \leq & 
\displaystyle
2 \pi e^{R/2} \int_0^R \int_0^{R-y_i} K
 e^{ -\frac12\max\{ y_i,y_j\} + \frac32\min \{ y_i,y_j\} - R}
 \left(\frac{\nu\alpha}{\pi}\right)^2 e^{-\alpha y_i} e^{-\alpha y_j} {\dd}y_j{\dd}y_i \\
 & = & 
 \displaystyle
 O\left(
 e^{-R/2}
 \int_0^{R/2} \int_t^{R-t} 
 e^{ -(\frac12+\alpha)s} e^{(\frac32-\alpha) t} {\dd}s{\dd}t
 \right)  \\
 & = & 
 \displaystyle
 O\left(
 e^{-R/2}
 \int_0^{R/2} 
 e^{ (1-2\alpha) t} {\dd}t
 \right) \\
 & = & 
 \displaystyle
 O\left( e^{-R/2} \right) = o(1).
\end{array} \]

\noindent
This proves part~\ref{itm:edges.i}.

\noindent
Let $B$ denote the number of pairs $\Xtil_i, \Xtil_j \in \Vtil$ for which $y_i+y_j \leq R/2$ and
$e^{\frac12(y_i+y_j)} < |x_i-x_j|_{\pi e^{R/2}} \leq \Delta(r_i,r_j)$.
Arguing as before, we see that

\[ \begin{array}{rcl}
\Ee B 
& \leq & 
\pi e^{R/2} \int_0^{R/2} \int_0^{R/2-t} K e^{\frac32(s+t)-R}
\left(\frac{\nu\alpha}{\pi}\right)^2 e^{-\alpha s} e^{-\alpha t} {\dd}s{\dd}t \\
& = & 
O\left( e^{-R/2} \int_0^{R/2} e^{(\frac32-\alpha)R/2} {\dd}t \right) \\
& = & 
O\left( R \cdot e^{(\frac12-\alpha)R/2} \right) \\
& = & 
o(1),
\end{array} \]

\noindent
which proves part~\ref{itm:edges.ii}.
\end{proof}

\begin{lemma}\label{lem:GiantAtLeastc}
A.a.s., $\GPo$ has a component containing at least $(c(\alpha,\nu)-o(1)) N$ vertices.
\end{lemma}

\begin{proof} 
We assume without loss of generality that $c(\alpha,\nu) > 0$ -- otherwise there is nothing to prove.
By part~\ref{itm:edges.i} of Lemma~\ref{lem:edges}, it suffices to show $\Gammatil$ has a component of the required size.
Let $\eps>0$ be arbitrary and choose $K = K(\eps)$ such that $\int_K^\infty \alpha e^{-\alpha y}{\dd}y < \eps/2$.
We now let $w = w(\alpha,\nu\alpha/\pi, K, \eps/2)$ be as provided by Lemma~\ref{lem:T}. 
We choose $w_1$ fixed but much larger than $w$ (to be determined later in the proof), and we set $h := R/2 + \log(\pi/2) - \log w_1$. 
(Observe that, this way we have ~$w_1 e^h = \frac{\pi}{2}e^{R/2}$.)
By Lemma~\ref{lem:Ch}, we have $\Pee_{\alpha, \lambda}( C_{w_1,h} ) = 1- o(1)$, as $h\to \infty$.

We now count the number $A$ of points $(x, y) \in \Vtil$ for which {\bf 1)} $y \leq K$, {\bf 2)} 
$|x| \leq (w_1-1-w)e^h = \frac{\pi}{2}e^{R/2}-(1+w)e^h$, and {\bf 3)} there is a path in $\Gammatil$ between $(x,y)$ and a point
in $[x-we^h, x+we^h]\times[h,2h]$ (that does not go outside of the box $[x-we^h, x+we^h]\times[0,2h]$).
Let us observe that, by Lemma~\ref{lem:cross}, if $C_{w_1,h}$ holds, then all the points counted by $A$ will belong to the same component.

For $(x,y) \in [-(w_1-1-w)e^h, (w_1-1-w)e^h]\times[0,K]$ let us define
$T_{x,y}$ as the event that {\bf 1)}, {\bf 2)}, {\bf 3)} hold
for $(x,y)$ with respect to the set of points $\{(x,y)\}\cup \Vtil$.
By Theorem~\ref{thm:palm} and Lemma~\ref{lem:T}, we have (with $\lambda = \nu \alpha /\pi$)

\begin{equation}\label{eq:EA}\begin{array}{rcl} 
\Ee A 
& = &
\displaystyle 
\int_0^K \int_{-(w_1-1-w)e^h}^{(w_1-1-w)e^h} \Pee_{\alpha, \lambda}( T_{x,y} ) \left(\frac{\nu\alpha}{\pi}\right) e^{-\alpha y}{\dd}x {\dd}y \\
& = & 
\displaystyle
2(w_1-1-w)e^h  \int_0^K \Pee_{\alpha, \lambda}( T(y; h, w ) ) \left(\frac{\nu\alpha}{\pi}\right) e^{-\alpha y}{\dd}y \\
& = & 
\displaystyle 
\pi e^{R/2}\left(1 - 2(w+1)/w_1\right)  
   \int_0^K \Pee_{\alpha, \lambda}( T(y; h, w ) ) \left(\frac{\nu\alpha}{\pi}\right) e^{-\alpha y}{\dd}y \\
& = & 
\displaystyle
\nu e^{R/2} \left(1 - \frac{2(w+1)}{w_1}\right)  \int_0^K \Pee_{\alpha, \lambda} ( T(y; h, w ) ) \alpha e^{-\alpha y}{\dd}y \\
& \geq & 
\displaystyle
N \left(1 - \eps/2 \right)  \int_0^K \Pee_{\alpha, \lambda} ( T(y; h, w ) ) \alpha e^{-\alpha y}{\dd}y \\
  \end{array} \end{equation}

\noindent
where $T(.;.,.)$ is as in Lemma~\ref{lem:T} and the last line holds provided we chose $w_1$ sufficiently large.
By Lemma~\ref{lem:T} and the choice of $K$, we see that

\[ \begin{array}{rcl} 
\Ee A 
& \geq & 
N (1-\eps/2) \int_0^K (\theta(y;\alpha,\lambda)-\eps/2)\alpha e^{-\alpha y}{\dd}y \\
& \geq &
N (1-\eps/2) \left( \int_0^\infty \theta(y;\alpha,\lambda)\alpha e^{-\alpha y}{\dd}y - \eps/2 \right) \\
& \geq  & 
N (c(\alpha,\nu)-\eps).    
   \end{array} \]

We now consider $\Ee A(A-1)$. Using Theorem~\ref{thm:palm}, we see that

\[ 
\Ee A(A-1) 
= 
\int_0^K \int_{-(w_1-1-w)e^h}^{(w_1-1-w)e^h} \int_0^K \int_{-(w_1-1-w)e^h}^{(w_1-1-w)e^h} 
\Pee_{\alpha, \lambda} ( T_{x,y} \cap T_{x',y'} ) \left(\frac{\nu\alpha}{\pi}\right)^2 e^{-\alpha y} e^{-\alpha y'} {\dd}x {\dd}y {\dd}x' {\dd}y'. 
\]

\noindent
Now we remark that $T_{x,y}$ and $T_{x',y'}$ are independent whenever $|x-x'| > 2we^h$.
This gives that 

\[ \begin{array}{rcl} 
\Ee A(A-1) 
& \leq &
\left( \Ee A\right)^2 + \Ee A \cdot \frac{4\nu}{\pi} w e^h \\
& \leq &
\left(\Ee A\right)^2 \left( 1 + O(\frac{w e^h}{\Ee A})\right) \\
& = &
\left(\Ee A\right)^2 \left( 1 + O(\frac{w e^h}{w_1 e^h})\right) \\
& \leq & 
\left(\Ee A\right)^2 \left( 1 + \eps \right), 
\end{array} \]

\noindent
where the last line holds provided we chose $w_1$ sufficiently large. 
Thus, we have $\Var(A) \leq (\eps+o(1))\left(\Ee A\right)^2$.
By Chebyschev's inequality we have

\[ \Pee_{\alpha, \lambda} ( A < (1-\eps^{1/4}) \Ee A ) \leq \frac{\eps+o(1)}{\sqrt{\eps}} = \sqrt{\eps} + o(1). \]

\noindent
Sending $\eps$ to zero finishes the proof.
\end{proof}

\begin{lemma}\label{lem:GiantAtMostc}
A.a.s., in $\GPo$ there are at least $(1-c(\alpha,\nu)-o(1))N$ vertices that are in components
of order $o(N)$.
\end{lemma}

\begin{proof}
Let $\eps > 0$ be arbitrary and choose $K = K(\eps)$ such that $\int_K^\infty \alpha e^{-\alpha y}{\dd}y < \eps/4$.
We now choose $n_0 = n_0(\alpha,\nu\alpha/\pi, K, \eps/4)$ according to Lemma~\ref{lem:U}.

Now, let $A$ denote the number of vertices $(x,y)$ of $\Gammatil$ such that {\bf 1)} $y\leq K$ and 
$x \in (-\frac{\pi}{2}e^{R/2}+e^{R/4}, \frac{\pi}{2} e^{R/2}-e^{R/4} )$ 
{\bf 2)} the component of $(x,y)$ in $\Gamma[ \Vtil \cap [x-e^{R/4}, x+e^{R/4}]\times[0,R/4]]$ 
has at most $n$ vertices and is contained in $[x-n,x+n]\times[0,n]$.

For $(x,y) \in (-\frac{\pi}{2}e^{R/2}+e^{R/4}, \frac{\pi}{2} e^{R/2}-e^{R/4} ) \times [0,K]$ let us 
denote by $U_{x,y}$ the event that the component of $(x,y)$ in $\Gamma[ \{(x,y)\} \cup \Vtil \cap [x-e^{R/4}, x+e^{R/4}]\times[0,R/4]]$
has at most $n$ vertices and is contained in $[x-n,x+n]\times[0,n]$.
Similarly to~\eqref{eq:EA}, we find:

\[\begin{array}{rcl} 
\Ee A 
& = &
\displaystyle 
\int_0^K \int_{-\frac{\pi}{2}e^{R/2}+e^{R/4}}^{\frac{\pi}{2} e^{R/2}-e^{R/4}} 
\Pee( U_{x,y} ) \left(\frac{\nu\alpha}{\pi}\right) e^{-\alpha y}{\dd}x {\dd}y \\
& = & 
\displaystyle
\nu e^{R/2} \left(1-O(e^{-R/4}) \right)  
\int_0^K \Pee_{\alpha, \nu\alpha/\pi}( U(y; n, R/4 ) ) \alpha e^{-\alpha y}{\dd}y \\
& \geq & 
\displaystyle
N (1 - o(1))  \int_0^K (1-\theta(y;\alpha,\nu\alpha/\pi)-\eps/4) \alpha e^{-\alpha y}{\dd}y \\
& \geq & 
N \left(1-c(\alpha,\nu)-\eps/2-o(1)\right),
  \end{array} \]

\noindent
where we used that $R/4 \to \infty$ in the third line.
Similarly to the proof of Lemma~\ref{lem:GiantAtLeastc}, we have

\[ 
\Ee A(A-1) 
= 
\int_0^K \int_{-\frac{\pi}{2}e^{R/2}+e^{R/4}}^{\frac{\pi}{2} e^{R/2}-e^{R/4}} 
\int_0^K \int_{-\frac{\pi}{2}e^{R/2}+e^{R/4}}^{\frac{\pi}{2} e^{R/2}-e^{R/4}} 
\Pee_{\alpha, \lambda} ( U_{x,y} \cap U_{x',y'} ) \left(\frac{\nu\alpha}{\pi}\right)^2 e^{-\alpha y} e^{-\alpha y'} 
{\dd}x {\dd}y {\dd}x' {\dd}y'. 
\]

\noindent
We remark that $U_{x,y}, U_{x',y'}$ are independent if $|x-x'| > 2e^{R/4}$.
This gives

\[  
 \Ee A(A-1) \leq \left(\Ee A\right)^2 + \Ee A \cdot \frac{2\nu}{\pi} e^{R/4}
 = \left(\Ee A\right)^2 (1 + o(1) ),
\]

\noindent
since $\Ee A = \Omega(N)$ and $e^{R/4} = O(\sqrt{N})$.
Applying Chebyschev's inequality we thus find that $A \geq (1-c(\alpha,\nu)-\eps)N$ a.a.s.

Now suppose that $(x,y) \in \Vtil$ satisfies {\bf 1)} and {\bf 2)} above, and pick let $(x',y') \in \left(\eR\setminus [x-e^{R/4},x+e^{R/4}]\right)
\times [0,R/4]$ and $(x'', y'') \in [-n,n]\times[0,n]$.
Then we have that $|x''-x'| \geq e^{R/4}-n > e^{n/2+R/8} \geq e^{(y+y')/2}$ for $N$ sufficiently large (using that $n$ is fixed and $R\to\infty$).
This shows that every point $(x,y)$ counted by $A$ belongs to a component of size $\leq n$, unless 
there is an edge between one of the $\leq n$ points of the component of $(x,y)$ in the graph induced by 
$\Vtil \cap [x-e^{R/4},x+e^{R/4}]\times[0,R/4]$ and a point with $y$-coordinate bigger than $R/4$.

To finish the proof, it thus suffices to show that the number of edges $B$ of $\GPo$ that join a
vertex with $y$-coordinate at least $R/4$ to a vertex of $y$-coordinate at most $n$ is $o_{p}(N)$ 
(using Lemma~\ref{lem:asymptotics}, and Theorem~\ref{thm:palm}).
To this end, it suffices to show that $\Ee B = o(N)$. The above claim, then, 
will follow from Markov's inequality. 

Thus, we compute, with $\Delta(.,.)$ as in Lemma~\ref{lem:asymptotics}:

\[ \begin{array}{rcl} 
\Ee B 
& = & 
\left(\frac{\nu\alpha}{2\pi}\right)^2 \pi e^{R/2} \int_0^n \int_{R/4}^R \Delta(y,y') 
e^{-\alpha(y+y')} {\dd}y {\dd}y' \\
& = & 
O\left( e^{R/2} \int_0^n \int_{R/4}^R e^{(\frac12-\alpha)(y+y')} {\dd}y {\dd}y'
+ 
e^{-R/2} \int_{R/4}^R e^{(\frac32-\alpha)(y+y')} {\dd}y {\dd}y' \right) \\
& = & 
O\left(
e^{R(\frac12 + \frac14(\frac12-\alpha))} + e^{R(1-\alpha)} 
\right) \\
& = & 
o( e^{R/2 } ) = o( N ),
\end{array}
\]

\noindent
using Lemma~\ref{lem:asymptotics} in the second line and $\alpha > \frac12$ in the last line.
By Markov's inequality, this gives that, $B=o_p(N)$.
Thus, at most $n \cdot o_{p}(N) = o_{p}(N)$ of the vertices counted by $A$ are usurped by long edges into large components.
\end{proof}

We have now finished the proof of our main result, as Lemma~\ref{lem:PoissonApproxIsOk}, Lemma~\ref{lem:GiantAtLeastc} 
and Lemma~\ref{lem:GiantAtMostc} together complete the proof of Theorem~\ref{thm:main}.

\section{Discussion}
We considered the emergence of the giant component in the  KPKBV model of random graphs on the hyperbolic plane. 
We show that the number of vertices in the largest component of $G(N;\alpha, \nu)$ satisfies a law of large numbers 
and converges in probability to a constant $c(\alpha, \nu)$. We give this function as the integral of the probability 
that a point percolates in an infinite continuum percolation model, for (almost) all values of $\alpha$ and $\nu$. 
When $\alpha = 1$, we show that there exists a critical value $\nuc$, such that when $\nu$ ``crosses'' $\nuc$, the giant 
component emerges with high probability. However, we do not know whether a giant component exists when $\nu = \nuc$.  
If the answer this question were negative then that would imply that $c(1,\nu)$ is continuous. 
We however conjecture that the answer is positive.

\begin{conjecture} 
$c(1,\nuc) > 0$.  
\end{conjecture}

Or equivalently, we conjecture that $\Gamma_{1,\lambdac}$ percolates. We have no particular reason to 
believe that this is the case except that the standard arguments showing non-percolation at criticality in other models do not seem to work, 
and the fact that we are dealing with a model with arbitrarily long edges and there are long-range percolation models that 
do percolate at criticality (cf.~\cite{AizenmanNewman}).

Another very natural question is for which values $(\alpha,\nu)$ 
the function $c(\alpha, \nu)$ is differentiable.

\section*{Acknowledgements}

We thank an anonymous referee for spotting an oversight in one of our proofs.

\bibliographystyle{plain}
\bibliography{ReferencesHyperGiantNew}

\appendix

\section{Explicit bounds on $\arccos(1-x)$ and $\sqrt{1+x}$}

For completeness, we spell out the derivation of some explicit bounds on $\arccos(1-x)$ and $\sqrt{1+x}$that we have used in the proof of 
Theorem~\ref{thm:main}.

\begin{lemma}\label{lem:sqrtlemma}
 For $x \in [-1,1]$, we have $1+\frac{x}{2}-100x^2 \leq \sqrt{1+x} \leq 1 + \frac{x}{2}$.
\end{lemma}

\begin{proof}
 The upper bound follows immediately from $(1+x/2)^2 = 1+x+x^2/4 \geq 1+x$.
 For the lower bound, observe that this certainly holds when $x \leq - \frac{1}{10}$ or $x\geq \frac{1}{10}$. 
 Squaring, we see that we need to show that, for all $-\frac{1}{10} < x < \frac{1}{10}$:
 
 \[ 1-x \geq 1 + x - (200-\frac14) x^2 - 100 x^3 + 10^4 x^4, \]
 
 \noindent
 which is equivalent to showing that, for all $-\frac{1}{10} \leq x <  \frac{1}{10}$:
 
 \[ 10^4 x^2 - 100 x - 200 + \frac14 \leq 0. \] 
 
\noindent
By convexity, it suffices to verify this only for $x=\pm\frac{1}{10}$.
(Which is easily seen to hold.)
\end{proof}

\begin{lemma}\label{lem:coslemma}
For all $x \in [0,1]$, we have $1-\frac{x^2}{2} \leq \cos(x) \leq 1 - \frac{x^2}{2} + \frac{x^4}{24}$. 
\end{lemma}

\begin{proof}
This can be easily seen from the Taylor expansion for $\cos$.
\end{proof}

\begin{lemma}\label{lem:arccoslemma}
For $x \in [0,1]$, we have $\sqrt{2x} \leq \arccos(1-x) \leq \sqrt{2x} + 1000 x^{3/2}$.
\end{lemma}

\begin{proof}
Let us write $y := \arccos(1-x)$.
Using Lemma~\ref{lem:coslemma} we have 

\[ 1-x = \cos(y) \geq 1 - \frac{y^2}{2}. \]
 
\noindent
Form this, the lower bound $y \geq \sqrt{2x}$ follows immediately.
For the upper bound we use that (again by Lemma~\ref{lem:coslemma}):

\[ 1-x \leq 1 - \frac{y^2}{2} + \frac{y^4}{24}, \]

\noindent
or in other words:

\[ \frac{y^4}{12} - y^2 + 2x \geq 0. \]

\noindent
By the quadratic formula, this is equivalent to:

\[ y^2 \leq 6(1 - \sqrt{1-\frac23 x}) \text{ or } y^2 \geq 6(1 + \sqrt{1-\frac23 x}). \]

Since the range of $\arccos$ is $[-1,1]$, we see that we must be in the first of these two cases.
Using Lemma~\ref{lem:sqrtlemma}, we see that

\[ y^2 \leq 6( 1 - (1-\frac{x}{3}-100\frac49 x^2)) = 2x+\frac{800}{3}x^2. \]

\noindent
Since $y \geq 0$, we see that

\[ 
y 
\leq 
\sqrt{2x} \cdot \sqrt{1+\frac{400}{3}x} 
  \leq
 \sqrt{2x} \cdot (1 + \frac{200}{3}x ) 
 \leq 
 \sqrt{2x} + 1000 x^{3/2},
\]

\noindent
where we have used Lemma~\ref{lem:sqrtlemma}.
\end{proof}

\end{document}